%% file: main.tex
\documentclass{amsart}

\usepackage{amsmath,amssymb,amsthm,amsrefs,a4wide}
\usepackage{latexsym}
\usepackage{indentfirst}
\usepackage{graphicx}
\usepackage{placeins}
\usepackage{booktabs}
\usepackage{algorithm}
\usepackage{algorithmic}
\usepackage{multirow}
\usepackage{subfig}
\usepackage{graphicx,color}
\usepackage{diagbox}

\newcommand{\bi}{\begin{itemize}}
\newcommand{\ei}{\end{itemize}}
\newcommand{\ben}{\begin{enumerate}}
\newcommand{\een}{\end{enumerate}}
\newcommand{\be}{\begin{equation}}
\newcommand{\ee}{\end{equation}}
\newcommand{\bea}{\begin{eqnarray}} 
\newcommand{\eea}{\end{eqnarray}}
\newcommand{\ba}{\begin{align}} 
\newcommand{\ea}{\end{align}}
\newcommand{\bse}{\begin{subequations}} 
\newcommand{\ese}{\end{subequations}}
\newcommand{\bc}{\begin{center}}
\newcommand{\ec}{\end{center}}
\newcommand{\bfi}{\begin{figure}}
\newcommand{\efi}{\end{figure}}

\newcommand{\bmp}[1]{\begin{minipage}{#1}}
\newcommand{\emp}{\end{minipage}}
\newcommand{\bp}{\begin{proof}}
\newcommand{\ep}{\end{proof}}

\newcommand{\mbf}[1]{{\mathbf #1}}

\newcommand{\R}{\mathbb{R}}

\newcommand{\pO}{{\partial\Omega}}
\newcommand{\dn}{{\partial_n}}

\newcommand{\gradientascent}{\textsc{GradientAscent}}
\newcommand{\vertexrefinement}{\textsc{VertexRefinement}}

\theoremstyle{plain}
\newtheorem{theorem}{Theorem}
\newtheorem{lemma}[theorem]{Lemma}
\newtheorem{prop}[theorem]{Proposition}
\newtheorem{conjecture}[theorem]{Conjecture}
\newtheorem{remark}[theorem]{Remark}

\usepackage[utf8]{inputenc}

\title[Peak heat flux conjecture]{The peak heat flux conjecture
  for the first Dirichlet eigenmode of convex planar domains}

\author[ZW, JH, MR, \and AB]{Zijian Wang, Jeremy G. Hoskins, Manas Rachh, \and Alex H. Barnett}
\thanks{%
(ZW) Department of Applied and Computational Mathematics, Yale University, New Haven, CT 06511, USA\\
\indent(JH) Department of Statistics and CCAM, University of Chicago, IL 06037, USA\\
\indent(MR) Department of Mathematics, Indian Institute of Technology Bombay, Powai, Mumbai 400076, India\\
\indent(AB) Center for Computational Mathematics, Flatiron Institute, New York, NY 10010, USA\\
\indent E-mail addresses: \url{zijian.wang@yale.edu}, \url{jeremyhoskins@uchicago.edu}, \url{mrachh@iitb.ac.in}, \url{abarnett@flatironinstitute.org}\\
\indent Date: \today
}

\begin{document}

\begin{abstract}
In this paper, we study the scale-invariant quantity
    \[\mathcal{G}(\Omega)=\frac{\|\partial_n u_1\|_{L^\infty(\partial\Omega)}}{\lambda_1},\]
    where $u_1$ is the first $L^2$-normalized Dirichlet Laplace eigenfunction of a Euclidean domain $\Omega$ and $\lambda_1$ is its eigenvalue.
    This is related to the peak boundary heat flux in the long time limit.
    For convex domains   
    we prove that
    $\|\partial_n u_1\|_{L^\infty(\partial\Omega)}$ is upper-bounded by a (domain-independent) constant multiple of $\lambda_1$.
Using layer potentials, we derive shape-derivative formulae for efficient gradient computations.
When combined with high-order Nystr\"om discretization, a fast boundary integral equation solver, and eigenvalue rootfinding, this
allows us to numerically optimize $\mathcal{G}$ over a class of rounded polygonal discretized domains.
Based on extensive numerical experiments, we then conjecture that, over the set of convex domains, $\mathcal{G}$ is maximized by the semidisk, with the peak flux at
the center of the diameter.
To lend analytical support to this conjecture, we prove that the semidisk is a critical point of $\mathcal{G}$ under infinitesimal perturbations of its circular arc.
\end{abstract}
\maketitle
\input{sections/introduction}
\input{sections/background}

\input{sections/proofs}
\input{sections/extremizer}
\input{sections/gradient}
\input{sections/numerical_experiments}
\input{sections/conclusion}
\section*{Acknowledgment}
We are grateful for discussions with Andrew Hassell, Charlie Epstein, and Stefan Steinerberger.
J.G.H. was supported in part by a Sloan Research Fellowship. 
The Flatiron Institute is a division of the Simons Foundation.

\clearpage
\newpage

\bibliographystyle{plain}  
\bibliography{main}

\end{document}

%% file: sections/introduction.tex
\section{Introduction}
We study the behavior of the Dirichlet eigenfunctions of a convex planar domain $\Omega$,
the nontrivial solutions to
\begin{equation}
  \begin{cases}
    \begin{aligned}
    -\Delta u &= \lambda u, \qquad \mbox{ in } \Omega,\\
    u &= 0,  \;\qquad \mbox{ on } \partial \Omega,
    \end{aligned}
    \end{cases}
    \label{eqn: bvp equation for u}
\end{equation}
where $\Delta$ is the two-dimensional Laplace operator.
It is well known that if $\Omega$ is open, connected, and bounded, then there exists an infinite sequence of Dirichlet eigenvalues $0<\lambda_1<\lambda_2\le \lambda_3 \le \dots,$ with associated $L^2(\Omega)$-normalized eigenfunctions $u_1,u_2,u_3,\dots\in L^2(\Omega)$.
Hassell and Tao showed that the $L^2$ norms of their boundary normal derivatives $\partial_n u_1,\partial_n u_2,\dots$ grow like $\sqrt{\lambda_k}$ as $k\to\infty$~\cite{hassell2002upper}.
Such eigenfunctions (the ``modes of a drum'') arise in a wide range of physical contexts, including heat flow, quantum mechanics, Brownian motion, acoustics, optics, and structural engineering.
However, in many applications, the first eigenfunction, i.e., the one with the smallest eigenvalue, determines the long-time behavior of the system. This motivates the following question:
\begin{center}
{\it How big can the normal derivative of $u_1$ be on the boundary of $\Omega$, and what shape, if any, maximizes it?}
\end{center}
Under rescaling of $\Omega$ by a linear factor, say $\alpha$,
then $\partial_n u_1$ scales like $\alpha^{-2}$,
and this is also true for $\lambda_1$.
Thus perhaps the simplest quantity that depends only on the shape (not the size)
of $\Omega$ is $\partial_n u_1/\lambda_1$.
One could then ask for the solution, or even the existence of solutions, to the following optimization problem:
$${\rm argmax}_{\Omega} \frac{\|\partial_n u_1\|_{L^\infty(\partial \Omega)}}{\lambda_1},$$
where the optimization is over all bounded convex domains $\Omega\subset\mathbb{R}^2$.
\begin{remark}[A physical interpretation]
  If $U(x,t)$ satisfies the homogeneous heat equation
  $\partial_t U = \Delta U$ in $\Omega$, with a zero-temperature condition on $\pO$,
  then in the long time limit $U(x,t) \sim c_1 e^{-\lambda_1 t} u_1(x)$,
  so that the distribution of heat flux through the boundary is proportional to $\partial_n u_1$.
  The question is then: what shape with a given decay rate pumps the most flux through some boundary
  point?%
  \footnote{To make a precise physical interpretation, one would
  need a set-up in which $c_1$ is constant.
  Constant initial conditions
  $U(\cdot,0)\equiv 1$, for example, give $c_1 = \|\partial_n u_1\|_{L^1(\pO)}/\lambda_1$
  and hence a different optimization problem.}
  This has a similar flavor to the recently solved
  \cite{hotspots}
  {\em hot spots conjecture}
  concerning the long-time behavior for convex domains instead
  with Neumann (insulating) boundary conditions.
\end{remark}

Despite its apparent simplicity, relatively little seems to be known about this problem. Here we prove a partial analytic result and provide numerical evidence motivating a conjecture on the extremizing shape. In particular, we establish the following upper bound.

\begin{figure}
\includegraphics[width=\textwidth]{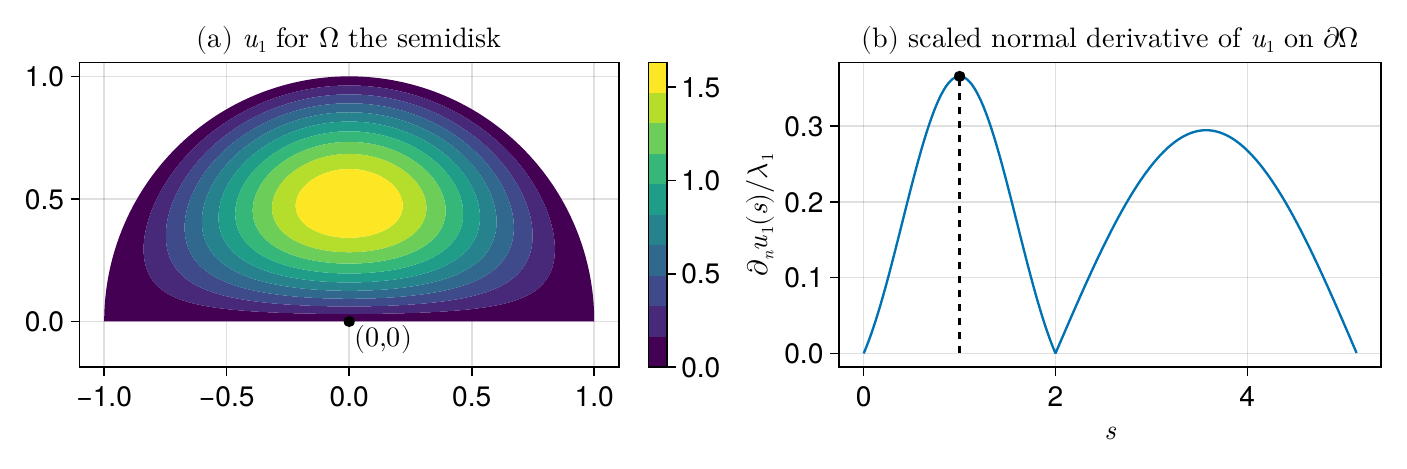}
\vspace{-2ex}
    \caption{(a) The $L^2(\Omega)$-normalized first Dirichlet eigenfunction $u_1$ of the unit semidisk, shown using the color scale to the right.
    (b) Its boundary derivative plotted vs $s$, the counterclockwise arc-length from the bottom-left corner $(-1,0)$ of $\Omega$ in panel (a). The boundary function is divided by $\lambda_1$ to make it scale-invariant.
    Its maximum over $s$, denoted by ${\mathcal G}(\Omega)$, is shown as a dot in (b),
    and occurs at $s=1$ (corresponding to the origin in panel (a)).
    Conjecture~\ref{conj: semicircle} is that there is no other convex domain $\Omega$ that exceeds this maximum of $C^*$.}
    \label{f:semidisk_u1}
\end{figure}

\begin{theorem}\label{thm: loose upper bound}
Let $\Omega$ be a bounded convex domain in the plane. If $\lambda_1$ is its first Dirichlet eigenvalue and $u_1$ the corresponding $L^2$-normalized eigenfunction then
\begin{align}
 \|\partial_n u_1 \|_{L^\infty(\partial \Omega)} \le C\lambda_1,
\end{align}
where $C$ is a constant independent of the domain.
\end{theorem}

The following conjecture is based on extensive numerical experiments.

\begin{conjecture}\label{conj: semicircle}
If we define $\mathcal{G}(\Omega):=\frac{\|\partial_n u_1\|_{L^{\infty}(\partial\Omega)}}{\lambda_1}$, then
  \begin{align}
    \sup_{\Omega\in\mathcal{D}}\mathcal{G}(\Omega)
    =\max_{\Omega\in\mathcal{D}}\mathcal{G}(\Omega)
    =C^*,
  \end{align}
where $\mathcal{D}$ is the space of all bounded convex domains in the plane, and the maximum is conjectured to be attained when $\Omega$ is a semidisk. The optimal constant $C^*$ is given by
  \[C^*:=\frac{1}{\sqrt{\pi}j_{1,1}|J_0(j_{1,1})|}\approx 0.3655840228073865,\]
where $J_{n}$ is the Bessel function of the first kind of order $n$, and $j_{n,i}$ is the $i$-th positive root of $J_n$.
\end{conjecture}

Fig.~\ref{f:semidisk_u1} illustrates this.
The conjecture is surprising for two reasons:
1) the extremal shape has corners, in contrast to the
extremizers for other known low-lying eigen-quantities
\cite{oudet04,osting10,antunes12},
and
2) the semidisk is also a shape that achieves
maximal growth in $\|\partial_n u_k\|_{L^\infty(\pO)}$ in the
{\em high-frequency} limit $k\to\infty$, and at the same point,
since the point is {\em self-focal} for rays under billiard
reflections in $\Omega$.
While this semiclassical ($k\to\infty$) concept of
self-focusing is irrelevant at $k=1$, the coincidence is intriguing.

Since $\|\partial_n u_1\|_{L^\infty(\partial\Omega)}$ is a maximum over boundary points and is invariant under rigid motions of the domain, we may translate and rotate $\Omega$ so that
a boundary point at which $\partial_n u_1$ is maximized lies at the origin $(0,0)$
and the inward unit normal points vertically upward (as in Fig.~\ref{f:semidisk_u1}).
We denote the set of all domains satisfying these properties as $\mathcal{D}_0$ and define the following function $\mathcal{F}:\mathcal{D}_0 \to \mathbb{R}^+$ by
  \begin{align}\label{eqn: definition of functional}
    \mathcal{F}(\Omega):=\frac{\partial_n u_1(0,0)}{\lambda_1},
  \end{align}
where $\partial_n u_1(\cdot)$ denotes the inward normal derivative of $u_1$. We can restate Conjecture \ref{conj: semicircle} as
  \begin{align}
    \sup_{\Omega\in\mathcal{D}_0} \mathcal{F}(\Omega)
    =\max_{\Omega\in\mathcal{D}_0} \mathcal{F}(\Omega)
    = |\sqrt{\pi}j_{1,1}J_0(j_{1,1})|^{-1}.
  \end{align}
As a first step towards this conjecture, the following theorem establishes that the semidisk is a critical point of the functional $\mathcal{F}$ if we restrict the variations to the circular arc of the boundary. In particular, we consider infinitesimal boundary perturbations which fix the bottom straight piece of the boundary, and obtain the following result.

\begin{theorem}\label{thm: critical point}
Let $V:[0,\pi] \to \mathbb{R}$ be a Lipschitz continuous function
with $V(0) = V(\pi) = 0$, and
define the parametric domain $$\Omega_t :=\{(r \cos \theta, r\sin \theta) \in \mathbb{R}^2\;|\, 0\le r < 1+t V(\theta), \, 0<\theta<\pi \},$$
where $t$ is a real parameter. Then, 
  \begin{align}
    \left.\frac{d}{dt}\right|_{t=0}{\mathcal{F}}(\Omega_t) = 0.
  \end{align}
\end{theorem}

\begin{remark}[Convexity]
  The reader may wonder why our conjecture restricts to convex domains.
  The answer is that the functional ${\mathcal G}(\Omega)$, equivalently
  ${\mathcal F}(\Omega)$,
  is unbounded when nonconvex domains are allowed.
  An example is the sector of angle $\beta>\pi$, namely in polars
  $\Omega_\beta := \{ (r,\theta)\in\R^2 \, | \, r<1, \, 0<\theta<\beta\}$,
  which has a {\em reentrant corner} at the origin.
  We have $u_1(r,\theta) = c J_\nu(j_{\nu,1}r) \sin \nu\theta$,
  where $\nu = \pi/\beta < 1$, so
  using $J_\nu(z) \sim (z/2)^\nu/\Gamma(\nu+1)$ as $z\to0$ we
  get
  that the normal derivative along $\theta=0$
  is $\partial_n u_1 = r^{-1} \partial_\theta u_1 \sim C r^{\nu-1}$,
  as $r\to 0$, which is unbounded.
  As $\beta\to\pi^+$, $\Omega_\beta$ becomes arbitrarily (Hausdorff) close
  to the semidisk,
  but ${\mathcal G}(\Omega_\beta)$ remains infinite.
  \end{remark}

The remainder of the paper is organized as follows: In Section~\ref{sec:background}, we review relevant bounds on eigenfunctions and boundary normal derivatives, as well as connections to classical extremal-domain problems and the shape-derivative viewpoint. In Section~\ref{sec:proofs}, we prove Theorem~\ref{thm: loose upper bound}, then, after gathering some general results on shape derivatives of eigenfunctions, Theorem \ref{thm: critical point}.  In Section~\ref{sec:extremizer}, we present analytic results and qualitative examples for a simple domain class, motivating the conjectured extremizer. In Section~\ref{sec:gradients}, we describe the boundary integral formulation and the shape-derivative calculations used to assemble gradients of the discretized objective. In Section~\ref{sec:numerics}, we report numerical experiments based on these gradients and discuss the observed extremizing shapes. We conclude in Section~\ref{sec:conclusion} with a summary and directions for future work.

%% file: sections/background.tex
\section{Background and related work}\label{sec:background}

\subsection{Bounds on eigenfunctions and boundary normal derivatives}

The study of Laplacian eigenfunctions and eigenvalues has a long history, reflecting their ubiquity in applications. Indeed, one could trace this history back at least to the 18th century and the analysis of waves on a string. Since then there has been a vast body of work on asymptotic and non-asymptotic behavior for a wide variety of boundary conditions, both in Euclidean space and on manifolds. Of particular interest has been the relationship to geometric features of the domain. Classic examples involving the eigenvalues include Weyl's law relating the growth of eigenvalues to the volume of the domain~\cite{Weyl1911}, and Kac's celebrated paper ``Can you hear the shape of a drum?''~\cite{kac1966can}.
A wonderful review of the geometric effects on eigenfunctions is provided by Grebenkov and Nguyen \cite{grebenkov2013geometrical}.

One vein of research of particular interest in applications is quantifying
various norms of eigenfunctions, and the locations of their maxima
or high-amplitude regions.  
Such work has a very different flavor for low-frequency eigenfunctions (such as $u_1$)
than it does in the high-frequency limit, where semiclassical and microlocal analysis are often involved.  

Most bounds on $u_1$ relate various domain norms.
Payne--Rayner showed that in $\mathbb{R}^2$,
$$\|u_1\|_{L^2(\Omega)} \le \frac{\sqrt{\lambda_1}}{\sqrt{4 \pi}}\|u_1\|_{L^1(\Omega)}$$
with equality for a disk \cite{payne1972isoperimetric}. This was subsequently extended to higher dimensions by the same authors \cite{payne1973some} and simplified by Kohler-Jobin \cite{kohler1977premiere} with the introduction of an auxiliary problem. 
Wang et al.\ further generalize the $\mathbb{R}^{n\geq 3}$ case to compact minimal surface with weakly connected Lipschitz boundary~\cite{wang2010isoperimetric}.
Work has been done to establish analogous bounds for other operators, e.g. $p$-Laplacians~\cite{alvino1998properties} and pseudo-Laplacians~\cite{mossino1983generalization}.

For $L_\infty$ norms, Payne and Stakgold showed that for convex planar domains~\cite{payne1973mean}
  \begin{align}
    \frac{\pi}{2|\Omega|}\|u_1\|_{L^1(\Omega)} \le \|u_1\|_{L^\infty(\Omega)} \,.
  \end{align}
Moreover, they proved a pointwise bound showing that the maximum of $u_1$ depends on its location in $\Omega$:
  \begin{align}
    u_1(x)\le d(x,\partial\Omega) \frac{\sqrt{\lambda_1}}{|\Omega|}\|u_1\|_{L^1(\Omega)}.
  \end{align}
Van den Berg obtained an explicit inradius-based $L^\infty$ bound for $L^2(\Omega)$-normalized Dirichlet eigenfunctions on bounded, open, connected domains in $\mathbb{R}^d$~\cite{van2000norm}.
A complementary line of research studies the location of maximizing point for the solutions to Schr\"{o}dinger equations (which specializes to Laplacian eigenfunctions when the potential $V$ is constant.)
Rachh and Steinerberger ~\cite{rachh2018location} established that for simply connected planar domains, the distance from the maximizing point to the boundary is lower bounded by $(|V|_\infty)^{-1/2}$ up to a universal constant factor, which translates to $\lambda^{-1/2}$ in the context of Laplacian eigenfunctions.
This result is extended to fractional Schr\"{o}dinger equations with appropriate exponents~\cite{biswas2017location}.
We refer the reader to \cite{grebenkov2013geometrical} for further extensions of these inequalities and related results.

One convenient tool for deriving such inequalities is {\it Dirichlet Green's function} for the domain, which satisfies the PDE
\begin{align}
\label{eq:dirichlet_greenfun}
\begin{cases}
-\Delta_x G(x,y) = \delta(x-y),\quad &x,y\in \Omega,\\
G(x,y) = 0, \quad &x \in \partial \Omega.
\end{cases}    
\end{align}
Using the identity
\begin{align}
  u_k(x) = \lambda_k \int_\Omega G(x,y) u_k(y)\,{\rm d}y \,,
\end{align}
Moler and Payne \cite{moler1968bounds} showed, among other things, that for Dirichlet eigenfunctions
\begin{align}
  |u_k(x)|\le \lambda_m\|u_k\|_{L^\infty(\Omega)}w(x),  \quad k=1,2,\dots
\end{align}
where $w(x) = \int_\Omega G(x,y)\,{\rm d}y$ is the {\it torsion function}, a much studied quantity in its own right dating back to the work of St Venant.

More recently, van den Berg showed that if $\Omega$ is an open, bounded and connected set in $d$ dimensions and $u_1$ is the $L^2(\Omega)$-normalized first Dirichlet eigenfunction, then
\begin{align}
\|u_1 \|_{L^\infty(\Omega)}
\le \frac{2^{\frac{2-d}{2}}}{\pi^{d/4}\Gamma(d/2)} \frac{\left( j_{\frac{d-2}{2},1}\right)^{\frac{d-2}{2}}}{\left|J_{\frac{d}{2}}\left(j_{\frac{d-2}{2},1}\right)\right|} \rho^{-d/2}\,,
\end{align}
where $\rho$ is the inradius~\cite{van2000norm}. In the plane ($d=2$), this reduces to 
\begin{align}
\|u_1 \|_{L^\infty(\Omega)} \le \frac{1}{\sqrt{\pi} J_1(j_{0,1})}\,\rho^{-1}.
\label{Berg2d}
\end{align}

We now turn to norm bounds that hold for all eigenfunctions.
The main question is how such norms grow with respect to $\lambda_K$,
as $k\to\infty$, and thus results tend to involve semiclassical or microlocal analysis.
The maximal bound on the sup norm was proved
in 1952 by Levitan \cite{levitan1952asymptotic} (see also Avakumovic \cite{avakumovic1956eigenfunktionen} and H\"{o}rmander \cite{hormander1968riesz}),
stating that
\begin{align}
  \|u_k\|_{L^\infty(\Omega)} \le c(\Omega) \lambda_k^{1/4}, \qquad k=1,2,\dots
  \label{maximal}
\end{align}
for some constant $c$ depending on $\Omega$ but independent of $k$.
This power of $\frac{1}{4}$ is sharp,
being achieved for instance by the disk:
the radially-symmetric subsequence
$u_{0,m}(r,\theta) = J_0(j_{0,m}r)/\sqrt{\pi}J_1(j_{0,m})$, $m=1,2,\dots$,
has a peak at $r=0$ growing as $\sqrt{j_{0,m}}$ from
large-argument asymptotics of $J_1$,
and one recalls $\lambda_{0,m} = j_{0,m}^2$.
Note that all rays launched from the disk center
return, and at equal times: it is a self-focus.
Indeed, in 2002, Sogge--Zelditch \cite{SZ02} showed that this maximal growth can
{\em only} happen with a self-focus.
For boundary derivatives of eigenfunctions, for $\pO$ smooth,
Hassell--Tao \cite{hassell2002upper} showed the $L^2$-norm behavior
\begin{align}
  c(\Omega)\lambda_k^{1/2} \le \|\partial_n u_k\|_{L^2(\partial \Omega)} \le C(\Omega) \lambda_k^{1/2},
  \qquad k=1,2,\dots.
\end{align}
One would expect dimensionally that the sup norms of the boundary derivatives
have one power of $\sqrt{\lambda_k}$ more than the sup norms of $u_k$ themselves
in \eqref{maximal},
and indeed the growth bound in 2D,
\begin{align}
  \|\partial_n u_k\|_{L^\infty(\pO)} \le C(\Omega) \lambda_k^{3/4},
  \qquad k=1,2,\dots,
  \label{maximaln}
\end{align}
is known to be sharp for $C^\infty$ domains,
a result summarized in \cite[Cor.~2.1]{Z03}
(although the smoothness is not made explicit there),
and proven for $C^\infty$ manifolds with $\pO$ concave
(this excludes planar domains) more recently by Sogge--Zelditch \cite[Cor.~1.2]{SZ17}.
Zelditch \cite[\S5.4]{Z03} mentions that
for planar domains \eqref{maximaln} it is achieved for the semidisk
or any (presumably smaller) sector of the disk.
Furthermore, paralleling results for the interior,
this sharp growth is known only to be achieved if there is a self-focal
boundary point \cite{Z03,SZ17}.

In this paper, we focus on the first eigenfunction and on the boundary quantity $\|\partial_n u_1\|_{L^\infty(\partial\Omega)}$, aiming to understand its largest possible size (after the scale-invariant normalization by $\lambda_1$).
Note that this problem is closely related in spirit to work on the torsion function and gradient bounds at ``failure points.'' In~\cite{hoskins2021towards}, numerical methods were used to propose candidates for extremizing domains in that context. Interestingly, for that problem the shape of the extremizer, if an extremizer exists at all, is still an open question.

\subsection{Extremal domains and eigenvalue problems}

Extremal domains related to spectral properties under fixed constraints have been well-explored due to their impact in physical science and engineering.
One of the most famous example is Lord Rayleigh's conjecture that the disk minimizes the first Dirichlet Laplacian eigenvalue within all planar domains of a fixed area~\cites{rayleigh1896theory}.
These isoperimetric problems are known for being easy to state yet hard to prove~\cite{henrot2006extremum}.
Almost $30$ years later, this was proved by Faber and Krahn ~\cite{krahn1925rayleigh} using a Schwarz rearrangement argument.
Recently, it has been strengthened to sharp quantitative bounds via Fraenkel asymmetry~\cites{brasco2015faber,brasco2012sharp}.

Numerous works have investigated extremal domains in a variety of settings.
One direction is to vary the underlying PDE problem.
For Neumann boundary conditions, the first Dirichlet eigenvalue is zero with the constant eigenfunction. 
Szeg{\"o} in two dimensions and Weinberger in $n$ dimensions proved that the ball is the minimizing shape for the second Neumann eigenvalue ~\cites{szego1954inequalities,weinberger1956isoperimetric}.
For Robin boundary conditions with positive boundary parameter, Bossel showed that the disk minimizes the first Dirichlet eigenvalue in $\mathbb{R}^2$ and Daners extended it into higher dimensions~\cites{bossel1986membranes,daners2006faber}.
For negative boundary parameters, Freitas et al.\ proved that the ball is a maximizer of the first eigenvalue when the boundary parameter is not too negative in $\mathbb{R}^2$~\cite{freitas2015first}.
Moreover, they showed that for sufficiently negative boundary parameter and general $\mathbb{R}^n$, the ball is not the maximizer.
This serves as the one of very first examples where the extremal domain is not a ball, in contrast to the majority of earlier Faber--Krahn type results.

Another direction of work is where the constraints for the admissible shapes are varied.
If we only consider polygons, P{\'o}lya and Szeg{\"o} proved that regular $n$-gon minimizes the first Dirichlet Laplacian eigenvalue when $n=3,4$ using a Steiner's symmetrization, which fails at $n=5$~\cite{polya2016isoperimetric}.
The case when $n=5$ remains an open problem.
Going beyond simply connected domains, Payne et al.\ showed that the annulus is the extremal domain within doubly connected domains, with Dirichlet/Neumann conditions on the outer/inner boundary respectively~\cite{payne1961some}.

Finally, one could vary the objective. P{\'o}lya showed, among other things, that the union of two disks minimizes the
second 
Dirichlet eigenvalue~\cite{polya1955characteristic} under fixed area.
When a convexity constraint is added, it is known that the optimal shape
contains two straight lines, but is not a stadium, although it is numerically close
to one (see Oudet \cite{oudet04}, who also minimized $\lambda_k$ at fixed area,
for $k=3,4,\dots,10$).
Kennedy extended P{\'o}lya's result to Robin eigenvalues with positive boundary parameters~\cite{kennedy2009isoperimetric}.
Rather than optimizing a specific eigenvalue, Ashbaugh et al.\ proved that the ball maximizes the {\em ratio} of the second to first Dirichlet eigenvalue \cite{ashbaugh1992sharp}.
Numerical optimization studies (using smooth radial Fourier series representations,
and where the eigenvalue problem was solved
by particular \cite{osting10} or fundamental solution \cite{antunes12}
methods) have found
the planar domains that extremize a variety of such small-index spectral ratios;
the domains seem to be smooth.
Inspired by the recently established gradient bound of the torsion function, as well as the nontrivial (conjectured) extremal domain for that problem~\cite{hoskins2021towards}, here we study in the maximum magnitude of the boundary normal derivative of the first Dirichlet eigenfunction.
Rather than enforcing an area constraint, we normalize by the first eigenvalue to make our objective scale-invariant.

\subsection{Shape derivatives via layer potentials}
Boundary integral equations provide an efficient framework for computing Laplacian eigenfunctions and thus objective functions like the one in~\eqref{eqn: definition of functional}, see~\cite{zhao2015robust}, for example. 
Layer potentials are fundamental building blocks for computing eigenfunctions using boundary integral equations~\cite{colton2013integral}. 
However, optimizing methods for~\eqref{eqn: definition of functional} tend to require gradient information for faster convergence,
which in turn requires Fr\'echet derivatives of
layer potentials.
In~\cite{potthast1994frechet}, Potthast established Fr\'echet differentiability of boundary integral operators in the context of inverse acoustic scattering, which provides theoretical support as well as tools for convergence analysis of iterated methods~\cites{hohage1998newton,hohage1997logarithmic}.
In particular, Potthast showed the differentiability of single and double layer potentials as they appear in the Brakhage-Werner formulation of the Helmholtz equation~\cites{brakhage1965dirichletsche,colton2013integral}.
Subsequently, the he also extended the results to Neumann boundary conditions~\cite{potthast1996frechet} the results of which are relevant for evaluating shape derivatives of~\eqref{eqn: definition of functional}.

%% file: sections/proofs.tex
\section{Proofs of Theorem \ref{thm: loose upper bound} and Theorem \ref{thm: critical point}}\label{sec:proofs}
The proof of Theorem~\ref{thm: loose upper bound} is a straightforward application of the Green's function method, combined with an estimate on the normal derivative of the torsion function
in convex domains.

\begin{proof}[Proof of Theorem \ref{thm: loose upper bound}]
We begin by observing that if $G$ is the Dirichlet Green's function for $\Omega$ defined in~\eqref{eq:dirichlet_greenfun}, then
\begin{align}
   u_1(x) = \lambda_1 \int_\Omega G(x,x') \,u_1(x')\,dx'.
\end{align}
Upon taking the gradient of both sides, we obtain
\begin{align}
    \nabla u_1(x) =\lambda_1 \int_\Omega \nabla_x G(x,x') u_1(x')\,dx'.
\end{align}
Without loss of generality, let $u_1 \ge 0.$
Furthermore, for $x \in \partial \Omega$ 
\begin{align}
    -n(x) \cdot \nabla u_1(x)\geq 0,\quad\quad 
    -n(x) \cdot \nabla G(x,x') \ge 0\,,
\end{align}
where $n(x)$  is the outward pointing normal.
It follows that
\begin{align}
    \|n\cdot \nabla u_1\|_{L^\infty(\partial \Omega)} \le \lambda_1 \|u_1\|_{L^\infty(\Omega)} \max_{x\in \partial \Omega}\int_\Omega n(x)\cdot \nabla_x G(x,x')\,dx'.
\end{align}
The integral on the right-hand side of the previous expression is the normal derivative of the torsion function $w$ for $\Omega,$ which itself satisfies the following PDE
\begin{align}
\begin{cases}
    \Delta w = -1, \quad & \mbox{ in } \Omega,\\
    \quad \,\,\,\, w=0,\quad & \mbox{ on } \partial \Omega.
\end{cases}
\end{align}
An inequality of Sperb (see Keady and McNabb as well) \cite{keady1993elastic,sperb1981maximum}
gives, for convex domains,
the following bound for the maximum of the gradient of the torsion function, 
\begin{align}
\|\nabla w\|_{L^\infty(\Omega)} \le \rho,    
\end{align}
where $\rho$ once again denotes the inradius of $\Omega$.
Upon substitution of this bound into the previous inequality, and using van den Berg's inequality \eqref{Berg2d}, we obtain
\begin{align}
     \|n\cdot \nabla u_1\|_{L^\infty(\partial \Omega)} \le \frac{1}{\sqrt{\pi} J_1(j_{0,1}) } \lambda_1, 
\end{align}
which completes the proof.
\end{proof}

We note that $C$ from the above proof is $C \approx 1.08676163613127$,
and thus about three times larger than the optimal $C^*$ discussed
in the introduction.


\subsection{Rate of change of eigenfunction normal derivative data with respect to
  shape deformations}
\label{s:deform}

To address Theorem \ref{thm: critical point},
we first need two lemmas for the rate of change of the normal derivative
data of the $i$th Dirichlet eigenmode $u_i$ of a domain
undergoing a general boundary deformation, which may be of independent interest.
In Section~\ref{s:appl},
these will then be applied to the mode $i=1$ of the semidisk to prove
the theorem.

Let $\Omega_t\subset\R^2$ be a domain with piecewise smooth boundary,
whose shape changes with a 
parameter $t\in\R$.
In particular, let $V$ be the outward normal ``velocity'' on $\pO$ with respect to
changing $t$, at $t=0$,
as in Theorem \ref{thm: critical point}.
Let $u_i$ be the $i$th normalized Dirichlet eigenmode of $\Omega_t$,
as in \eqref{eqn: bvp equation for u}.
The classical Hadamard formula for the eigenvalue rate of change is
\be
\dot{\lambda_i} = -\int_\pO V (\dn u_i)^2 ds,
\label{had}
\ee
where we use the standard notation that a dot above a quantity denotes
$\partial_t$, i.e., its partial with respect to the deformation parameter.
Here $ds$ is the arc-length element on $\pO$.
For this, see Grinfeld \cite{grin10}. 
He derives formulae (his Sec.~1.4) for
eigenfunction changes only in cases where the resulting
formula is local on the boundary; however, the normal-derivative data
case that we need below is not covered.

Recall the calculus of moving surfaces (e.g., see Grinfeld \cite[\S3.2]{grin10}).
For a function $f$ defined in the closure of $\Omega$
there are two types of time-derivative at $x\in\pO$:
the {\em total} derivative $D_t f(x) := (d/dt)f(x(t))$
in which $x(t)$ moves with the boundary
(this is often written $D/Dt$, or $\delta/\delta t$ in Grinfeld),
and the usual partial $\partial_t f(x)= \dot{f}(x)$ in which
the spatial location $x$ is held fixed as $t$ changes.
They are related by the chain rule \cite[(42)]{grin10}:
\be
D_t f = \dot{f} + V \dn f \qquad \mbox{ on } \pO.
\label{chain}
\ee
As an application, if $u_i$ is the ($t$-dependent)
Dirichlet eigenfunction of the deforming domain, then
$u_i|_\pO = 0$ for all $t$, so that \eqref{chain}
gives (see \cite[(70)]{grin10}) the partial derivative
\be
\dot{u_i} = - V \dn u_i  \qquad \mbox{ on } \pO.
\label{dotdat}
\ee
The {\em Thomas rule}
(e.g.\ see Lemma~\ref{lem: geometric_variations})
gives the rate of change of the unit normal vector $n(x)$
at a point $x\in\pO$ moving with the normal deformation $V$,
due to the local rotation rate of the surface,
\be
D_t n = -\nabla_\pO V,
\label{thomas}
\ee
where $\nabla_\pO$ is the surface gradient operator,
equal to $\tau \partial_s$ in our case of $d=2$ dimensions,
$\tau$ being the unit tangent vector in the increasing $s$ sense.
We will also need the local mean curvature of $\pO$ at $t=0$,
namely $\kappa=1/R$, where $R$ is the signed local radius of curvature
(the sign is positive for the circle; note that \cite[p.664]{grin10} uses the opposite sign).

\begin{lemma}  
  Let $u_i$ be a Dirichlet eigenfunction of a domain $\Omega$ whose
  $t$-dependent shape change is specified by the boundary normal velocity
  function $V$. Then the total and partial derivatives of its boundary data are
  related by
  \be
  D_t (\dn u_i) = \dn\dot{u_i} - V \kappa \dn u_i  \qquad \mbox{ on } \pO.
  \label{b}
  \ee
  \label{l:dt}
\end{lemma}
\begin{proof}
We follow \cite[Sec.~4.2]{grin10} but for the Dirichlet case. 
The product rule for the total derivative gives
$$
D_t (\dn u_i) = D_t(n\cdot\nabla u_i)|_\pO
= n\cdot D_t \nabla u_i + (D_t n)\cdot \nabla u_i.
$$
Applying \eqref{chain} to each component in the first $D_t$, and \eqref{thomas} to the second $D_t$,
gives
\be
D_t (\dn u_i) =  n\cdot\nabla\dot{u_i} + V n\cdot\nabla\nabla u\cdot n - \nabla_\pO V \cdot\nabla u_i~,
\label{a}
\ee
where $\nabla \nabla$ denotes the Hessian matrix.
For a general Helmholtz solution $(\Delta+\lambda)u=0$ in $\R^d$ we have \cite[(74)]{grin10}
$$
 n\cdot\nabla\nabla u\cdot n = -\lambda u - \Delta_\pO u
- \kappa \dn u, \qquad \mbox{ on } \pO,
$$
where $\Delta_\pO$ is the surface Laplace-Beltrami operator (noting
the sign change from Grinfeld).
This can be proved easily in the $d=2$ case via local polar coordinates.
Since $u_i$ additionally has zero Dirichlet boundary condition this simplifies to
$$
n\cdot\nabla\nabla u_i\cdot n =  -\kappa \dn u_i~.
$$
Applying this to \eqref{a}, and noting that $\nabla u_i$ is purely normal,
gives \eqref{b}.
\end{proof}

We next show that $\dot u_i$ solves a BVP in the static (fixed)
domain $\Omega$.

\begin{lemma}  
  Let $u_i$ be a simple normalized Dirichlet eigenfunction of
  the bounded domain $\Omega$ undergoing shape change with boundary velocity
  function $V$, as above.
  Then its rate of change $\psi:= \dot u_i$ is the unique
  solution to the constrained 
  inhomogeneous Dirichlet Helmholtz interior BVP
  \bea
  (\Delta + \lambda_i)\psi &=& -\dot{\lambda_i} u_i =: f \qquad\mbox{ in } \Omega
  \label{pde}
  \\
  \psi &=& -V \dn u_i =: g \qquad\mbox{ on } \pO
  \label{bc}
  \\
  (\psi,u_i)_{L^2(\Omega)} &=& 0 ~.
  \label{cond}
  \eea
  \label{l:bvp}
\end{lemma}
\begin{proof}
  The PDE comes from the taking $\partial_t$ of
  $(\Delta+\lambda_i)u_i = 0$.
  The boundary condition is simply \eqref{dotdat}.
  By themselves this pair \eqref{pde}--\eqref{bc} is an on-resonance
  (secular)
  inhomogeneous Dirichlet Helmholtz BVP, whose null-space is thus the span of
  $u_i$.
  Thus a solution $\psi$ exists if and only if the data $f$ and $g$ satisfy
  a standard compatibility condition
  (see, e.g., \cite[Thm.~4.10]{mclean2000strongly}).
  This condition comes from Green's identity
  applied to the data and the null-space vector $u_i$ (i.e., solution to
  the homogeneous adjoint BVP),
  $$
  \int_\Omega [\psi (\Delta + \lambda_i)u_i - u_i (\Delta + \lambda_i) \psi]
  dx
  = \int_\pO (\psi \dn u_i - u_i \dn \psi) ds.
  $$
  Applying \eqref{pde}, using
  $(\Delta+\lambda_i)u_i = 0$, and using that $u_i=0$ on $\pO$,
  the condition is
  $$
  (u_i, f)_{L^2(\Omega)} + \int_\pO g \dn u_i ds = 0~.
  $$
  For $f$ and $g$ defined in \eqref{pde}--\eqref{bc},
  since $u_i$ is normalized,
  the condition holds by the Hadamard formula \eqref{had}.
  This assures existence of a (nonunique) solution to the BVP
  \eqref{pde}--\eqref{bc}.
  The orthogonality condition \eqref{cond}, which makes $\psi$ unique,
  comes from applying the product rule to the $t$-derivative of the normalization $\int_\Omega u_i^2 dx = 1$
  to give $(u_i,\dot u_i)_\Omega = 0$.
  (See \cite[(68)]{grin10}; a subtlety is that the $t$-dependent
  fundamental theorem of calculus for volume integrals is needed, which
  in the Dirichlet case
  does not introduce any new terms.)
\end{proof}

Finally, it is convenient to split the solution of \eqref{pde}--\eqref{cond}
into a particular solution $\tilde \psi$ plus a homogeneous solution
$w$. Namely, if $\tilde \psi$ solves
\be
(\Delta + \lambda_i)\tilde \psi \;=\; -\dot{\lambda_i} u_i
\qquad \mbox{ in } \Omega
\label{ps}
\ee
with arbitrary boundary condition, and then $w$ is the unique solution to
  \bea
  (\Delta + \lambda_i)w &=& 0 \qquad\mbox{ in } \Omega
  \label{wpde}
  \\
  w &=& -V \dn u_i - \tilde\psi \qquad\mbox{ on } \pO
  \label{wbc}
  \\
  (w,u_i)_{L^2(\Omega)} &=& -(\tilde\psi, u_i)_{L^2(\Omega)}~,
  \label{wcond}
  \eea
  then $\psi=\tilde\psi + w$ solves \eqref{pde}--\eqref{cond},
  and is thus the desired eigenfunction partial $\dot{u_i}$.

\subsection{Application to $u_1$ for the semidisk:
    proof of Theorem \ref{thm: critical point}}
\label{s:appl}

The first Dirichlet eigenfunction of the semidisk $0<r<1$, $0<\theta<\pi$, is,
in polar coordinates,
\be
u_1(r,\theta) \;=\; \frac{-2}{\sqrt{\pi} J_1'(k_1)} J_1(k_1 r) \sin\theta~,
\label{u1}
\ee
where the normalization constant (negative so that $u_1$ is nonnegative)
follows from, eg, \cite[(10.22.37)]{dlmf},
and its eigenvalue is $\lambda_1 = k_1^2 = j_{1,1}^2$.

\begin{prop}  
  Let $\Omega$ be the semidisk, $u_1$ and $\lambda_1$
  be as above, then for any $\dot{\lambda_1}\in\R$,
\be
\tilde\psi(r,\theta) = \frac{-\dot{\lambda_1}}{\sqrt{\pi}k_1J_1'(k_1)}
rJ_0(k_1 r)
\sin\theta
\label{psit}
\ee
satisfies the $i=1$ case of \eqref{ps}, i.e., is a particular solution.
\label{p:psit}
\end{prop}
\begin{proof}
The Laplacian in polar coordinates is
$$
\Delta = \frac{1}{r}\partial_r(r \partial_r \,\cdot\,) + \frac{1}{r^2}\partial_{\theta\theta}.
$$
By separation of variables, $\tilde\psi (r,\theta) = U(r) \sin\theta$
where $U$ solves the inhomogeneous radial ODE
\be
L_1[U](r) := \frac{1}{r}(r U')' - \frac{1}{r^2} U(r) + \lambda_1 U(r)
\;=\;
\frac{2\dot{\lambda_1}}{\sqrt{\pi} J_1'(k_1)} J_1(k_1 r)
~,
\label{ode}
\ee
and vanishes at $r=0$ so that $\tilde\psi$ is regular at the origin.
By direct differentiation,
$L_1[rJ_0(k_1 r)] = 2\partial_r J_0(k_1 r) = -2k_1 J_1(k_1 r)$,
so that
$U(r) = -\dot{\lambda_1}/(\sqrt{\pi}k_1J_1'(k_1)) \cdot r J_0(k_1 r)$
satisfies \eqref{ode}.
\end{proof}

We now specialize to a boundary normal deformation function $V(\theta)$ on the
semidisk, and which vanishes on the straight part of $\pO$,
as in Theorem~\ref{thm: critical point}.
In this case, since $\partial_n u_1(\theta) = -(2k_1/\sqrt{\pi})\sin\theta$
on the semidisk $r=1$, and $ds = d\theta$, Hadamard \eqref{had} becomes
\be
\dot{\lambda_1} = \frac{-4k_1^2}{\pi}\int_0^\pi V(\theta)\sin^2 \theta\, d\theta~,
\label{had1}
\ee
which may be inserted into \eqref{psit} to define $\tilde\psi$.
It then only remains to solve the homogeneous on-resonance BVP
\eqref{wpde}--\eqref{wcond}; for this we use what
is (perhaps confusingly)
known as the {\em method of particular solutions}.

\begin{prop}  
  With $V$ an arbitrary smooth function on $0<\theta<\pi$,
  recalling \eqref{had1} and $\tilde\psi$ given by \eqref{psit},
  the unique $i=1$ homogeneous solution to
  \eqref{wpde}--\eqref{wcond} takes the form of the Fourier-Bessel sine series
\be
w(r,\theta) = \sum_{\ell=1}^\infty c_\ell J_\ell(k_1r)\sin \ell \theta,
\label{mps}
\ee
for some coefficients $c_\ell\in\R$, and where, in particular, $c_1=0$.
\label{p:mps}
\end{prop}
\begin{proof}
  Each term in \eqref{mps} solves the Helmholtz equation \eqref{wpde} in the semidisk
  with zero boundary condition on the straight bottom boundary.
  Thus each $c_\ell$, $\ell=2,3,\dots$ is uniquely determined by
  matching Fourier sine series coefficients of the boundary
  data on the right-hand side of \eqref{wbc}.
  Uniqueness here holds because $J_\ell(k_1) \neq 0$, $\ell=2,3,\dots$;
  see \cite[\S 10.21(i)]{dlmf}.
  The formula for these $c_\ell$ are not hard to write, but are not needed.
  Yet, because $J_1(k_1)=0$, $c_1$ is undetermined by \eqref{wbc}, and is
  instead
  fixed by \eqref{wcond}.
  Inserting \eqref{psit} and \eqref{u1}, dropping prefactors for simplicity,
  one sees that the right-hand side of \eqref{wcond} is a multiple
  of
  $$
  \int_0^\pi \int_0^1 \bigl[
    rJ_0(k_1 r)\sin\theta \cdot J_1(k_1 r) \sin\theta
    \bigr] rdrd\theta =
  \frac{\pi}{2} \int_0^1 r^2 J_0(k_1 r)J_1(k_1 r) dr = 0,
  $$
  using a Bessel function identity \cite[(10.22.7)]{dlmf} with $\mu=0$ and $\nu=1$.
  Thus $(w,u_1)_{L^2(\Omega)} = 0$, so $c_1=0$.
\end{proof}

We finally apply all of the above results.
Let $\mbf{0}$ denote the origin of the semidisk; it is easy to check using
\eqref{u1} that this point is the global maximum of $\dn u_1$ on $\pO$.
Since there is no surface deformation at the origin,
$D_t \dn u_i(\mbf{0}) = \dn \dot{u_i}(\mbf{0})$, by Lemma~\ref{l:dt}.
Then by the quotient rule, to show $d{\mathcal F}/dt =0$ as claimed
in Theorem \ref{thm: critical point},
it is sufficient to show the vanishing of
\be
\lambda_1 \dn \dot{u_1}(\mbf{0}) - \dot{\lambda_1} \dn u_1(\mbf{0}),
\label{quot}
\ee
for an arbitrary deformation $V$ on the semidisk.
However, by Lemma~\ref{l:bvp},
\[
\dn \dot{u_1}(\mbf{0}) = \dn \psi(\mbf{0})
= \dn \tilde\psi(\mbf{0}) + \dn w(\mbf{0})
 = \frac{-\dot{\lambda_1}}{\sqrt{\pi}k_1 J_1'(k_1)} + 0
\]
where we used Prop.~\ref{p:psit} for the first term, and see that the second
term vanishes by Prop.~\ref{p:mps}, since $J'_\ell(0)=0$, $\ell=2,3,\dots$
while $c_1=0$.
Yet using \eqref{u1} and $J_0(0)=1$ shows that the second term of \eqref{quot}
is
\[
- \dot{\lambda_1} \cdot \frac{-2k_1}{\sqrt{\pi}J_1'(k_1)}.
\]
Recalling $\lambda_1=k_1^2$ shows that the two terms in \eqref{quot} cancel,
completing the proof of the theorem.

%% file: sections/extremizer.tex
\section{Analytic results for rectangular domains}\label{sec:extremizer}
Before introducing the numerical support for the conjectured extremizer, we first examine $\mathcal{F}$ in \eqref{conj: semicircle} (and the resulting optimal constant) when $\mathcal{D}$ is restricted to rectangles. The purpose of this brief analysis is two-fold: a) firstly to illustrate the invariance of $\mathcal{F}$ under scaling, and b) to demonstrate that the most symmetric shape in the class does not optimize the function, thus hinting that the circle need not necessarily optimize $\mathcal{F}$ over all convex domains.

Let $R_{a,b}$ be a rectangle with side lengths $a\geq b>0$. Recall that its first eigenfunction and eigenvalue are given by
    \[u(x,y) = \sin\frac{\pi x}{a}\sin\frac{\pi y}{b},
    \quad\quad \lambda = \pi^2(a^{-2}+b^{-2}).\]
Direct computation shows that the Euclidean norm of its gradient achieves maximum at boundary:
\begin{equation}
  \begin{aligned}
  |\nabla u|^2
  &= \pi^2\left(a^{-2}\cos^2\frac{\pi x}{a}\sin^2\frac{\pi y}{b} + b^{-2}\sin^2\frac{\pi x}{a}\cos^2\frac{\pi y}{b}\right)\\
  &= \frac{\pi^2}{a^2b^2}\left(a^2X-(a^2+b^2)XY+b^2Y \right),
  \end{aligned}
  \end{equation}
where $X:=\sin^2\frac{\pi x}{a}$ and $Y:=\sin^2\frac{\pi y}{b}$ take values in $[0,1]$. This, together with the assumption that $a\geq b$, implies that
  \[\|\nabla u\|_{L^\infty(\partial\Omega)} = \frac{\pi}{b}.\]
Another straightforward computation shows that
  \[\|u\|_{L^2(\Omega)}^2
    = \int_0^a\int_0^b\sin^2(\pi x/a)\sin^2(\pi y/b)dydx
    = \frac{ab}{4}.
    \]
Combining the above computations, we have
  \[\phi(R_{a,b}) 
  = \frac{\frac{\pi}{b}}{\pi^2(a^{-2}+b^{-2})\,\sqrt{\frac{ab}{4}}}
  = \frac{2}{\pi}\left((a^{-2}+b^{-2})\, a^{1/2}\,b^{3/2}\right)^{-1}.\]
	By comparing the exponents, we see that $\phi(R_{ra,rb}) = \phi(R_{a,b})$ for any scaling factor $r>0$. Let $\alpha := \frac{a}{b}\geq 1$ be the side ratio of the rectangle. Using scale-invariance, we obtain
	  \[\phi(R_{a,b}) = \phi(R_{\alpha, 1})
	  = \frac{2}{\pi}\frac{1}{(\alpha^{-2}+1)\,\alpha^{1/2}}
	  = \frac{2}{\pi}\,\frac{\alpha^{3/2}}{\alpha^2+1}.\]
	  Finally, it follows that $\alpha = \sqrt{3}$ is the optimal rectangle, for which $\mathcal{F} \approx 0.362794816$.

%% file: sections/gradient.tex
\section{Shape derivatives and gradient computation}
\label{sec:gradients}

This section describes an integral equation based approach for evaluating $\mathcal{F}$ and its shape derivatives.  Our starting point is a boundary integral characterization of Dirichlet eigenpairs (a related approach is described in~\cite{zhao2015robust}), followed by the computation of the shape derivatives of $\mathcal{F}$ extending the results from~\cite{potthast1994frechet,hohage1998newton}.

\subsection{Boundary integral equations and layer potentials}
Let $\Omega\subset\mathbb{R}^2$ be a bounded $C^2$ domain with boundary $\partial \Omega$.
Let $k>0$ and suppose $u\neq 0$ solves the Helmholtz equation
  \begin{align}
    (\Delta + k^2)u = 0\ \text{in }\Omega,\qquad u=0\ \text{on }\partial \Omega,
    \label{homhelm}
  \end{align}
so that $k^2$ is a Dirichlet eigenvalue.
Let $x,y\in\partial \Omega$, let $dS$ denote arc-length measure, and let $n(x)$ denote the outward unit normal.
The free-space Green's function for $(\Delta+k^2)$ in $\mathbb{R}^2$ is given by
  \begin{align}
    G_k(x,y) := \frac{i}{4}H_0^{(1)}(k|x-y|),
  \end{align}
  where $H_{0}^{1}(z)$ is the Hankel function of the first kind of order $0$.
Consider the following layer potential operators defined on the boundary $\partial \Omega$
\begin{equation}
  \begin{aligned}
    S_k[\sigma](x)
    &:= \int_{\partial \Omega} G_k(x,y)\,\sigma(y)\,ds_y,\\
    S'_k[\sigma](x)
    &:= {\rm p.v.}\int_{\partial \Omega} \partial_{n_x}G_k(x,y)\,\sigma(y)\,ds_y,\\
    S''_k[\sigma](x)
    &:= {\rm f.p.}\int_{\partial \Omega} \partial_{n_x}^2G_k(x,y)\,\sigma(y)\,ds_y,\\
    D_k[\sigma](x)
    &:= {\rm p.v.}\int_{\partial \Omega} \partial_{n_y}G_k(x,y)\,\sigma(y)\,ds_y,\\
    D'_k[\sigma](x)
    &:= {\rm f.p.}\int_{\partial \Omega} \partial_{n_x}\partial_{n_y}G_k(x,y)\,\sigma(y)\,ds_y,
  \end{aligned}
  \label{eqn: layer_potentials}
  \end{equation}
with ${\rm p.v.}$ denoting Cauchy principal value and ${\rm f.p.}$ denoting Hadamard finite part~\cite{mclean2000strongly}. In a slight abuse of notation, we use $S_{k}[\sigma](x)$ to also denote the operator mapping $\sigma$ on the boundary $\partial \Omega$ to a solution of the Helmholtz equation in $\Omega$.

Using Green's identities, the Dirichlet eigenfunction $u$ can be represented via the single-layer potential $S_k$ with density $\sigma:=-\partial_n u|_{\partial\Omega}$, i.e. $u = S_{k}[\partial_{n} u|_{\partial \Omega}]$.
  Using standard jump conditions~\cite{colton2013integral} and taking the limit
  \begin{equation}
  \sigma(x) = \partial_{n}u(x) = \lim_{h \to 0^{+}} n(x) \cdot \nabla u(x - h n(x)) \,,
  \end{equation}
  we get that $\sigma$ satisfies the equation
    \begin{align}
\label{eqn: bie_eig_sprime}
    \frac{1}{2}\sigma(x) - S_k'[\sigma](x) = 0,
    \qquad x \in \pO.
    \end{align}
Following the procedure in~\cite{zhao2015robust}, the Dirichlet eigenvalue can be obtained by solving the nonlinear eigenvalue problem of determining $k$ such that $I/2 - S'_{k}$ has a nontrivial null space. 

\subsection{Curve parametrization and normal perturbations}
Let $\gamma:[0,L]\to\mathbb{R}^2$ parameterize $\partial \Omega$ and let $\tau$ denote the unit tangent, $n$ the outward unit normal, and $\kappa$ the signed curvature.
In this section we denote the smooth scalar deformation field by $\nu$ on $\partial \Omega$, consider the normal perturbation
  \begin{align}
    \gamma_\varepsilon(s) := \gamma(s) + \varepsilon n(s)\,\nu(s),
  \end{align}
thus $\varepsilon$ denotes what was called $t$ in Sec.~\ref{s:deform}.
We denote by $\delta(\cdot)$ the derivative at $\varepsilon=0$, i.e.
  \begin{align}
    \delta f := \partial_\varepsilon|_{\varepsilon=0} f.
  \end{align}
Let $s$ be an arbitrary smooth parameter on $\partial \Omega$, so that
  \begin{align}
    dS = |\dot\gamma(s)|\,ds,\qquad \tau := \frac{\dot\gamma}{|\dot\gamma|}.
  \end{align}
For simplicity, for functions and operators defined on $\partial\Omega$ (such as layer potentials, $\sigma,n,\tau,\kappa,\nu$), we write $f(s)$ to denote their values at point $\gamma(s)$.
For a scalar function $\nu$ on $\partial \Omega$, we write $\frac{d\nu}{d\tau}:=\partial_\tau \nu$ for its tangential derivative.
We will use the following standard geometric identities (see, e.g., \cite{delfour2011shapes} for a more comprehensive treatment).
\begin{lemma}\label{lem: geometric_variations}
Under the normal perturbation $\gamma_\varepsilon=\gamma+\varepsilon n \nu$, we have
  \begin{align}
    \delta(|\dot\gamma|) = \kappa \nu|\dot\gamma|,
    \qquad \mbox{ and } \quad \delta n = -\tau\,\frac{d\nu}{d\tau}
    \quad \mbox{(Thomas rule)}.
  \end{align}
\end{lemma}
\begin{proof}
Let $R$ denote clockwise rotation by $\pi/2$ and define the unit normal by $n:=|\dot\gamma|^{-1}R\dot\gamma$ (so $R\tau=n$ and $Rn=-\tau$).
Differentiating $\gamma_\varepsilon=\gamma+\varepsilon n \nu$ with respect to $s$ gives $\partial_s\gamma_\varepsilon=\dot\gamma+\varepsilon(\dot n\,\nu+n\,\dot \nu)$, so
  \begin{align}
    \delta|\dot\gamma|
    = \frac{\dot\gamma}{|\dot\gamma|}\cdot(\dot n\,\nu+n\,\dot \nu)
    = \tau\cdot\dot n\,\nu
    = \kappa \nu\,|\dot\gamma|.
  \end{align}
For the normal variation, differentiating $n=|\dot\gamma|^{-1}R\dot\gamma$ yields
  \begin{align}
    \delta n
    = -\frac{\delta|\dot\gamma|}{|\dot\gamma|^2}R\dot\gamma + \frac{1}{|\dot\gamma|}R(\dot n\,\nu+n\,\dot \nu)
    = -\tau\,\frac{d\nu}{d\tau},
  \end{align}
using $R\dot\gamma=|\dot\gamma|n$, $R\dot n=\kappa|\dot\gamma|n$, and $\frac{d\nu}{d\tau} = \dot \nu/|\dot\gamma|$.
\end{proof}

\subsection{Shape derivative and variations of layer potentials}
Let $x_\ast=(0,0)\in\partial \Omega$ denote the boundary point at which we evaluate the normal derivative used for the objective function $\mathcal{F}$, as defined in \eqref{eqn: definition of functional}.
The boundary integral equation \eqref{eqn: bie_eig_sprime} determines $\sigma$ (and hence $u$) only up to scale. Thus, we express the functional in terms of the unnormalized eigenfunction as
  \begin{align}\label{eqn: objective_from_sigma}
    \mathcal{F}(\Omega)
    = \frac{\partial_nu(x_\ast)}{k^2\,\sqrt{\mathcal{N}(\Omega)}}
    = \frac{\sigma(x_\ast)}{k^2\,\sqrt{\mathcal{N}(\Omega)}},
  \end{align}
where $\mathcal{N}(\Omega):=\|u\|_{L^2(\Omega)}^2$.
Since we only consider perturbations fixing $x_\ast$ and the local boundary orientation, we have $\delta(\sigma(x_\ast))=(\delta\sigma)(x_\ast)$.
Differentiating \eqref{eqn: objective_from_sigma} gives
  \begin{align}\label{eqn: dF_formula}
    \delta\mathcal{F}
    = \frac{(\delta\sigma)(x_\ast)}{k^2\,\sqrt{\mathcal{N}}}
      -\frac{2(\delta k)\sigma(x_\ast)}{k^3\,\sqrt{\mathcal{N}}}
      -\frac{\sigma(x_\ast)\,\delta\mathcal{N}}{2k^2\,\mathcal{N}^{3/2}}.
  \end{align}
Thus we require $\delta k$, $\delta\sigma$, and $\delta\mathcal{N}$.
The variations $\delta k$ and $\delta\sigma$ follow from differentiating the boundary integral equation \eqref{eqn: bie_eig_sprime}.
Since $S'_k$ depends on both the boundary and the wavenumber $k$, differentiating produces two operator variations: the shape derivative $\delta(S'_k)$ and the derivative with respect to the wavenumber $k$ denoted by $(\delta k)\,\partial_k S'_k$.
  \begin{align}
    0 = \Big(\frac{1}{2} I - S'_k\Big)[\delta\sigma] - \delta(S'_k)[\sigma] - (\delta k)\,(\partial_k S'_k)[\sigma].
  \end{align}
Formulas for $\delta(S'_k)$ and $\partial_k S'_k$ are shown in the lemma below and can be obtained via standard calculations~\cite{liu2020regularized}.
\begin{lemma}\label{lem: layer_variations}
Let $\nu$ be smooth on $\partial \Omega$ and consider the perturbation $\gamma_\varepsilon=\gamma+\varepsilon n \nu$. Defining
  \begin{align}
    S_{\tau,k}[\sigma](x)
    := {\rm p.v.}\int_{\partial \Omega} \partial_{\tau_x}G_k(x,y)\,\sigma(y)\,ds_y,
  \end{align}
we have 
  \begin{align}\label{eqn: delta_sprime}
    \delta\left(S'_k\right)[\sigma]
    &= \nu\,S''_k[\sigma] + D'_k[\nu\sigma] - \frac{d\nu}{d\tau}\,S_{\tau,k}[\sigma] + S'_k[\kappa \nu\sigma]\\
    \partial_kS_k'[\sigma]&=\frac{ik}{4}\int_{\partial\Omega}H_0^1(k|x-y|)(x-y)\cdot n(x)\sigma(y)\,ds_y.
  \end{align}
\end{lemma}
\begin{proof}
We first compute $\partial_kS_k'$.
Straightforward differentiation of the kernel gives us
  \begin{align}\label{eqn: partial_nx G_k}
    \partial_{n_x}G_k(x,y)
    =\frac{ik}{4}H_1^1(k|x-y|)\frac{(x-y)\cdot n_x}{|x-y|}.
  \end{align}
Using identity~\cite[Eq.~10.6.6]{dlmf}, we have
  \begin{align}
    \partial_k \left[kH_1^1(k|x-y|)\right]=|x-y|H_0^1(k|x-y|).
  \end{align}
Substituting this into \eqref{eqn: partial_nx G_k} gives the desired result.
Next, we compute $\delta\left(S'_k\right)[\sigma]$.
Fix a smooth parametrization $\gamma:[0,L]\to\partial \Omega$. 
For simplicity, we write $S'_k[\sigma](t)$ for the operator evaluated at the target point $\gamma(t)$ (and similarly for $D'_k,S''_k,S_{\tau,k}$). 
The perturbed kernel is
  \begin{align}
    S'_{k,\varepsilon}[\sigma](t)
    := {\rm p.v.}\int_0^L n_\varepsilon(t)\cdot\nabla_x G_k\big(\gamma_\varepsilon(t),\gamma_\varepsilon(s)\big)\,\sigma(s)\,|\dot\gamma_\varepsilon(s)|\,ds.
  \end{align}
Straightforward computation gives us
\begin{equation}
\label{eqn: three_integrals}
  \begin{aligned}
    \delta\left(S'_k\right)[\sigma](t)
    &= \int_0^L \Big(\nu(t)\,\partial_{n_x}^2G_k(\gamma(t),\gamma(s)) + \nu(s)\,\partial_{n_x}\partial_{n_y}G_k(\gamma(t),\gamma(s))\Big)\sigma(s)\,|\dot\gamma(s)|\,ds\\
    &\quad\ +\int_0^L \big(\delta n(t)\cdot\nabla_x G_k(\gamma(t),\gamma(s))\big)\sigma(s)\,|\dot\gamma(s)|\,ds\\
    &\quad\ +\int_0^L \partial_{n_x}G_k(\gamma(t),\gamma(s))\,\sigma(s)\,\delta|\dot\gamma(s)|\,ds,
  \end{aligned}
  \end{equation}
Applying Lemma~\ref{lem: geometric_variations}, the last two integrals become
\begin{equation}
  \begin{aligned}
    \int_0^L \big(\delta n(t)\cdot\nabla_x G_k(\gamma(t),\gamma(s))\big)\sigma(s)\,|\dot\gamma(s)|\,ds
    &= -\frac{d\nu}{d\tau}(t)\,S_{\tau,k}[\sigma](t)\\
    \int_0^L \partial_{n_x}G_k(\gamma(t),\gamma(s))\,\sigma(s)\,\delta|\dot\gamma(s)|\,ds&=S'_k[\kappa \nu\sigma]
  \end{aligned}
  \end{equation}
Rewriting the first integral in \eqref{eqn: three_integrals} in terms of $S''_k$, $D'_k$ yields \eqref{eqn: delta_sprime}.
\end{proof}

Although $S''_k$ and $D'_k$ are hypersingular, the combination in \eqref{eqn: delta_sprime}
has a kernel that is bounded on the diagonal.
To see this, we first separate the leading order singularity of the Helmholtz Green's function by writing:
\begin{align}
  G_k(x,y) = -\frac{1}{2\pi}\log\rho + \frac{k^2}{8\pi}Q_k(\rho) + R_k(\rho),
  \quad\quad Q_k(\rho):=\rho^2\log\rho,
  \quad\quad \rho:=|x-y|
\end{align}
where $R_k$ is $C^2$. Fix the target parameter $t$ and set
\[
r:=t-s,\qquad \gamma_r:=\gamma(t)-\gamma(s),\qquad \nu_r:=\nu(t)-\nu(s).
\]
Then as $r\to 0$,
\begin{align}
  \gamma_r = r\,\dot\gamma(t) + \mathcal{O}(r^2),\qquad
  |\gamma_r|^2 = r^2|\dot\gamma(t)|^2 + \mathcal{O}(r^3),\qquad
  \nu_r = r\,\dot\nu(t) + \mathcal{O}(r^2),
\end{align}
and $\dot\nu(t)=|\dot\gamma(t)|\,\frac{d\nu}{d\tau}(t)$ by definition of the
tangential derivative. Moreover, since $\gamma_r$ is tangent to first order,
$n(t)\cdot\gamma_r=\mathcal{O}(r^2)$, and hence
\[
\partial_{n_x}\log|\gamma_r| = \mathcal{O}(1).
\]

The leading order singularity in \eqref{eqn: delta_sprime} arises from
differentiating the singular part $-\frac{1}{2\pi}\log|x-y|$ twice in normal directions,
which produces a $|\gamma_r|^{-2}$ singularity. At the level of leading terms, one has
\begin{align}
  \partial_{n_x}^2\log|\gamma_r| = \frac{1}{|\gamma_r|^2}+\mathcal{O}(1),\qquad
  \partial_{n_x}\partial_{n_y}\log|\gamma_r| = -\frac{1}{|\gamma_r|^2}+\mathcal{O}(1),
\end{align}
so the first two terms in \eqref{eqn: delta_sprime} contribute, up to bounded remainders,
\begin{equation}
\begin{aligned}\label{eqn: leading_order_calc}
  -\frac{1}{2\pi}\Big(\nu(t)\,\partial_{n_x}^2\log|\gamma_r|
  +\nu(s)\,\partial_{n_x}\partial_{n_y}\log|\gamma_r|\Big)
  &= -\frac{1}{2\pi}\frac{\nu(t)-\nu(s)}{|\gamma_r|^2} + \mathcal{O}(1) \\
  &= -\frac{1}{2\pi}\frac{\nu_r}{|\gamma_r|^2} + \mathcal{O}(1) \\
  &= -\frac{1}{2\pi}\frac{1}{r|\dot\gamma(t)|}\frac{d\nu}{d\tau}(t)
     + \mathcal{O}(1).
\end{aligned}
\end{equation}
On the other hand,
\begin{align}
  \partial_{\tau_x}\log|\gamma_r|
  = \frac{\tau(t)\cdot\gamma_r}{|\gamma_r|^2}
  = \frac{1}{r|\dot\gamma(t)|}+\mathcal{O}(1),
\end{align}
so the normal-rotation term $-\frac{d\nu}{d\tau}(t)\,S_{\tau,k}[\sigma]$
cancels the $\frac{1}{r}$ singularity coming from $\nu(t)S''_k[\sigma]+D'_k[\nu\sigma]$.
The remaining pieces (including $S'_k[\kappa\nu\sigma]$) are bounded, since they involve
only first derivatives of $\log|\gamma_r|$ and the $C^2$ remainder $R_k$. 
Note that $Q_k''(\rho)$ is $\mathcal{O}(\log\rho)$, which is bounded when combined with the $\nu(t)-\nu(s)$ factor, following essentially the same calculations as in \eqref{eqn: leading_order_calc}. 
Thus the kernel
of the combination in \eqref{eqn: delta_sprime} is bounded on the diagonal.

\subsection{Computing $\delta k$ and $\delta\sigma$}
Let $k=k_\varepsilon$ and $\sigma=\sigma_\varepsilon$ satisfy \eqref{eqn: bie_eig_sprime} on the perturbed boundary $\partial \Omega_\varepsilon$.
Differentiating \eqref{eqn: bie_eig_sprime} yields
  \begin{align}\label{eqn: bie_linearized}
    \Big(\frac{1}{2} I - S'_k\Big)[\delta\sigma]
    = \delta(S'_k)[\sigma] + (\delta k)\,(\partial_k S'_k)[\sigma].
  \end{align}
Let $\mu$ be a left nullvector of $\frac{1}{2} I - S'_k$ and normalize it by $\langle \mu,\sigma\rangle_{L^2(\partial \Omega)}=1$.
Taking the $L^2(\partial \Omega)$ inner product of \eqref{eqn: bie_linearized} with $\mu$ and using $(\frac{1}{2} I - S'_k)^*\mu=0$ yields
  \begin{align}
    0
    = \big\langle \mu, \delta(S'_k)[\sigma]\big\rangle_{L^2(\partial \Omega)}
      + (\delta k)\,\big\langle \mu,(\partial_k S'_k)[\sigma]\big\rangle_{L^2(\partial \Omega)}.
  \end{align}
Assuming $\langle \mu,(\partial_k S'_k)[\sigma]\rangle_{L^2(\partial \Omega)}\neq 0$ (e.g. if the first eigenvalue is simple), we obtain
  \begin{align}\label{eqn: dk_formula}
    \delta k
    = -\frac{\langle \mu,\delta(S'_k)[\sigma]\rangle_{L^2(\partial \Omega)}}{\langle \mu,(\partial_k S'_k)[\sigma]\rangle_{L^2(\partial \Omega)}}.
  \end{align}
Once $\delta k$ is known, equation \eqref{eqn: bie_linearized} determines $\delta\sigma$ up to adding a multiple of $\sigma$.
To fix this, we add the rank one integral operator $\mu\langle\sigma,\cdot\rangle$ that imposes the orthogonality constraint
  \begin{align}\label{eqn: dsigma_orthog}
    \langle \sigma,\delta\sigma\rangle_{L^2(\partial \Omega)} = 0,
  \end{align}
which yields a unique solution of \eqref{eqn: bie_linearized}.

\subsection{Normalization and derivative of the objective}
To evaluate $\delta\mathcal{N}$ in \eqref{eqn: dF_formula}, we use
a standard identity due to Rellich \cite{rellich} that
for any Dirichlet eigenfunction \eqref{homhelm} with eigenvalue $k^2$,
  \begin{align}\label{eqn: rellich_identity}
    2k^2\int_\Omega u^2\,dx
    = \int_{\partial \Omega} (\partial_n u)^2\,(x\cdot n)\,ds.
  \end{align}
Then, recalling $\sigma = -\partial_n u$, we have
  \begin{align}\label{eqn: uL2_from_sigma}
    \mathcal{N}(\Omega)=\|u\|_{L^2(\Omega)}^2
    = \frac{1}{2k^2}\int_{\partial \Omega} \sigma^2\,(x\cdot n)\,ds.
  \end{align}
Differentiating \eqref{eqn: uL2_from_sigma} and using Lemma~\ref{lem: geometric_variations} together with
  \begin{align}
    \delta(x\cdot n) = \delta x\cdot n + x\cdot\delta n = \nu - (x\cdot \tau)\,\frac{d\nu}{d\tau},
  \end{align}
we obtain
  \begin{align}\label{eqn: dN_formula}
    \delta\mathcal{N}
    = -2\frac{\delta k}{k}\,\mathcal{N}
      + \frac{1}{k^2}\int_{\partial \Omega} \sigma\,\delta\sigma\,(x\cdot n)\,ds
      + \frac{1}{2k^2}\int_{\partial \Omega} \sigma^2\Big(\nu - (x\cdot \tau)\,\frac{d\nu}{d\tau} + \kappa \nu(x\cdot n)\Big)\,ds.
  \end{align}
Combining \eqref{eqn: dF_formula}, \eqref{eqn: dk_formula}, \eqref{eqn: bie_linearized}, and \eqref{eqn: dN_formula} gives the derivative of $\mathcal{F}$ under normal perturbation $\nu$.

\subsection{Non-normal perturbations}
In this subsection, we consider a non-normal deformation field $V(s)$.
The perturbed curve becomes
\begin{align}
  \gamma_\varepsilon(s):=\gamma(s)+\varepsilon V(s).
\end{align}
Note that $V(s)=n(s)\nu(s)$ corresponds to the normal perturbation scenario that we studied earlier.
Following the notation used in Lemma \ref{lem: geometric_variations}, we obtain the derivatives of $|\dot\gamma|$ and $n$ using similar calculations.
\begin{lemma}\label{lem: geometric_variations2}
  Under the non-normal perturbation $\gamma_\varepsilon(s):=\gamma(s)+\varepsilon V(s)$, we have
    \begin{align}
      \delta(|\dot\gamma|)=\frac{\dot\gamma\cdot\dot V}{|\dot\gamma|},\quad\quad
      \delta n=\frac{R\dot V}{|\dot\gamma|}-n\frac{\dot\gamma\cdot\dot V}{|\dot\gamma|^2}.
    \end{align}
\end{lemma}
Next, we compute the derivative of $S_k'$.
\begin{lemma}
  The derivative of $S_k'$ is given by
  \begin{equation}
    \begin{aligned}
      \delta(S_k')[\sigma](t) &= \int_0^L (U_1(t,s)+U_2(t,s)+U_3(t,s))\sigma(s)|\dot\gamma(s)|
      \,ds\\
      U_1(t,s):&=R\frac{dV}{d\tau}(t)\cdot\nabla_xG_k(\gamma(t),\gamma(s))\\
      U_2(t,s):&=n(t)\cdot\left(\nabla_x^2G_k(\gamma(t),\gamma(s))(V(t)-V(s))\right)\\
      U_3(t,s):&=\partial_{n_x}G_k(\gamma(t),\gamma(s))\left(\tau(s)\cdot\frac{dV}{d\tau}(s)-\tau(t)\cdot\frac{dV}{d\tau}(t)\right).
    \end{aligned}
  \end{equation}
    
\end{lemma}
\begin{proof}
  Recall that 
    \begin{align}
      S_k'[\sigma](t)=\text{p.v.}\int_0^L 
      (n(t)\cdot A(t,s))
      \sigma(s)|\dot\gamma(s)|\,ds,\quad\quad A(t,s):&= \nabla_x G_k(\gamma(t),\gamma(s))
    \end{align}
  Using the product rule, we have 
    \begin{align}
      \delta(S_k')[\sigma](t)
      &=\int_0^L \left( 
      (\delta n(t))\cdot A(t,s)
      + n(t)\cdot\delta(A(t,s))
      + (n(t)\cdot A(t,s))\frac{\delta|\dot\gamma(s)|}{|\dot\gamma(s)|}
      \right)\sigma(s)|\dot\gamma(s)|\,ds.
    \end{align}
  By straightforward calculations and an application of the chain rule, we have
    \begin{align}
      \partial_\varepsilon\nabla_xG_k(\gamma_\varepsilon(t),\gamma_\varepsilon(s))
      =\left(\nabla_x^2G_k(\gamma_\varepsilon(t),\gamma_\varepsilon(s))\right)\partial_\varepsilon\gamma_\varepsilon(t)
      +\left(\nabla_y\nabla_xG_k(\gamma_\varepsilon(t),\gamma_\varepsilon(s))\right)\partial_\varepsilon\gamma_\varepsilon(s).
    \end{align}
  Evaluating at $\varepsilon=0$ gives us
    \begin{align}
      \delta A(t,s) = \nabla_x^2G_k(\gamma(t),\gamma(s))(V(t)-V(s)).
    \end{align}
  This, together with Lemma \ref{lem: geometric_variations2}, gives the desired result.
\end{proof}
We need to establish boundedness of the kernel $\delta(S_k')$ on the diagonal.
The argument is essentially the same as the one used in the normal case.
In the non-normal case, the $1/r$ cancellation is between $U_1$ and $U_2$.
We define $r$ and $\gamma_r$ as before, and let $V_r=V(t)-V(s)$.
We consider the logarithmic part of the free-space Green's function.
Straightforward computations give us 
  \begin{align}
    \nabla_x\log |\gamma_r|=\frac{\gamma_r}{|\gamma_r|^2},\quad\quad \nabla_x^2\log|\gamma_r|=\frac{I}{|\gamma_r|^2}-\frac{2\gamma_r\gamma_r^T}{|\gamma_r|^4},
  \end{align}
where $T$ denotes vector transpose and $I$ is the identity matrix.
For $U_1$, we have
  \begin{align}\label{eqn: nonnormal cancel 1}
    \frac{R\dot V}{|\dot\gamma|}\cdot\frac{\gamma_r}{|\gamma_r|^2}
    =\frac{R\dot V\cdot(r\dot\gamma+\mathcal{O}(r^2))}{r^2|\dot\gamma|^3}
    =-\frac{\dot V\cdot n}{r|\dot\gamma|^2}+\mathcal{O}(1),
  \end{align}
where the last equality uses the definition of the rotation $R$.
For $U_2$, we have 
  \begin{align}
    n^T\left(\frac{I}{|\gamma_r|^2}-\frac{2\gamma_r\gamma_r^T}{|\gamma_r|^4}\right)V_r
    =\frac{n\cdot V_r}{|\gamma_r|^2}-\frac{2(n\cdot\gamma_r)(\gamma_r\cdot V_r)}{|\gamma_r|^4}.
  \end{align}
Since $n\cdot\dot\gamma=0$, both the numerator and the denominator of the second term are $\mathcal{O}(r^4)$.
Moving to the first term, we have 
  \begin{align}\label{eqn: nonnormal cancel 2}
    \frac{n\cdot V_r}{|\gamma_r|^2}
    =\frac{rn\cdot\dot V+\mathcal{O}(r^2)}{r^2|\dot\gamma|^2+\mathcal{O}(r^3)}
    =\frac{n\cdot\dot V}{r|\dot\gamma|}+\mathcal{O}(1).
  \end{align}
This cancels with the $1/r$ singularity in \eqref{eqn: nonnormal cancel 1}.
Same as the normal case, it is straightforward to verify boundedness of the remaining terms.
To rearrange $\delta S_k'$ in layer potentials, it is convenient to decompose $V$ into normal and tangential contributions:
    \begin{align}
        V=\alpha n+\beta \tau.
    \end{align}
Straightforward computation gives us 
    \begin{align}
    \delta S_k'[\sigma] = \left(-\frac{d\alpha}{d\tau}S_{\tau,k}[\sigma]+S_k'[\kappa\alpha\sigma]+\alpha S''_k[\sigma]+D'[\alpha\sigma]\right)
    +\left(\beta\frac{d}{d\tau}S'_k[\sigma]-S_k'\left[\frac{d\beta}{d\tau}\sigma\right]\right),
  \end{align}
where the first part agrees with the normal case.
Finally, we differentiate the normalization term.
\begin{lemma}
  The derivative of the normalization term $\mathcal{N}$ is given by
  \begin{align}
    \delta \mathcal{N} = 
    -2\frac{\delta k}{k}\mathcal{N}
    +\frac{1}{k^2}\int_{\partial\Omega} \sigma\delta\sigma(x\cdot n)\,ds
    +\frac{1}{2k^2}\int_{\partial\Omega}\sigma^2\left[V\cdot n+x \cdot\left(R\frac{dV}{d\tau}\right)\right]ds
  \end{align}
\end{lemma}
\begin{proof}
  We start with the Rellich identity stated in \eqref{eqn: rellich_identity}.
  Recall that
    \begin{align}
      \mathcal{N} = \frac{1}{2k^2}\int_0^L\sigma(s)^2B(s)\,ds,\quad\quad B(s):=(\gamma(s)\cdot n(s))|\dot\gamma(s)|.
    \end{align}
  An application of the product rule gives us 
    \begin{align}\label{eqn: deltaN for nonnormal}
      \delta\mathcal{N}
      =-2\frac{\delta k}{k}\mathcal{N}+\frac{1}{2k^2}\int_0^L\left[2\sigma(s)\delta\sigma(s) B(s) + \sigma(s)^2\delta B(s)
      \right]\,ds.
    \end{align}
  In the normal perturbation case, we formulate $\delta k$ and $\delta\sigma$ using $\delta(S_k')$.
  Therefore, we do not need to recompute these two terms.
  It remains to compute $\delta B$, which follows from the identities in Lemma \ref{lem: geometric_variations2}.
  Observe that
  \begin{equation}
    \begin{aligned}
      \delta B
      &=(\delta\gamma\cdot n)|\dot\gamma|+(\gamma\cdot\delta n)|\dot\gamma|+(\gamma\cdot n)\delta|\dot\gamma|\\
      &=|\dot\gamma|(V\cdot n) 
      +\gamma\cdot\left(\frac{R\dot V}{|\dot\gamma|}-n\frac{\dot\gamma\cdot\dot V}{|\dot\gamma|^2}\right)
      +(\gamma\cdot n)\frac{\dot\gamma\cdot\dot V}{|\dot\gamma|}\\
      &=(V\cdot n)|\dot\gamma|+\gamma\cdot R\dot V.
    \end{aligned}
  \end{equation}
  Substituting this back into \eqref{eqn: deltaN for nonnormal} completes the proof.
\end{proof}
\begin{remark}
\label{rem:grad-num}
In our numerical experiments, we approximate a domain $\Omega$ using a rounded polygon with $N$ vertices $(r_i\cos\theta_i,r_i\sin\theta_i)_{i=1,2,\dots N}$.
To perform optimization on $\mathcal{F}$, we require the gradient with respect to the radial parameters $r_i$.
We start by evaluating $\delta\mathcal{F}$ for $2N$ non-normal perturbations, corresponding to the translation of each vertex in the $x$ and $y$ directions.
For each vertex, we project its $x$ and $y$ derivatives onto the $(\cos\theta_i,\sin\theta_i)$ direction to obtain the derivative with respect to $r_i$.
\end{remark}

\subsection{Fast evaluation of shape derivatives}
The evaluation of the shape derivative of $\mathcal{F}$ requires the evaluation of the variations $\delta k$, $\delta \sigma$, and $\delta \mathcal{N}$. 
In order to evaluate $\delta k$ we need to compute $\delta S_{k}'[\sigma]$, and $\partial_{k} S_{k}'[\sigma]$. The standard Helmholtz fast multipole method (FMM) in two
dimensions can be used for the fast evaluation of $\delta S_{k}'[\sigma]$. However, 
the evaluation of $\partial_{k} S_{k}'[\sigma]$ requires small modifications as the kernel does not satisfy the Helmholtz equation.
A simple calculation shows that
\begin{equation}
\partial_{k} S_{k}'[\sigma] = \frac{ik}{4}\int_{\partial \Omega}  (x-y)\cdot n(x) H_{0}^{1}(k|x-y|)\sigma(y) \, ds\,.
\end{equation}
This kernel can be evaluated using three Helmholtz FMM computations. Let 
\begin{equation}
\begin{aligned}
\phi_{1}(x)&= \frac{ik}{4}\int_{\partial \Omega}  y_{1} H_{0}^{1}(k|x-y|)\sigma(y) \, ds\,,\\
\phi_{2}(x)&= \frac{ik}{4}\int_{\partial \Omega}  y_{2} H_{0}^{1}(k|x-y|)\sigma(y) \, ds\,,\\
\phi_{3}(x)&= \frac{ik}{4}(x \cdot n(x))\int_{\partial \Omega}  H_{0}^{1}(k|x-y|)\sigma(y) \, ds\,,
\end{aligned}
\end{equation} 
each of which is computable using the standard Helmholtz FMM. Then
\begin{equation}
\partial_{k} S_{k}'[\sigma](x) = \phi_{3}(x) - n_{1}(x) \phi_{1}(x) - n_{2} \phi_{2}(x) \,. 
\end{equation}

Turning to the computation of $\delta \sigma$, we need to solve an integral equation with the operator $1/2I - S_{k}'$. As observed in Remark~\ref{rem:grad-num}, the variation $\delta \mathcal{F}$ is computed one perturbation
direction at a time. Moreover, the computation of the objective function $\mathcal{F}$ also requires solving the same integral equation with different data. Thus when computing $\delta \mathcal{F}$ for domains described by $N$ radii, $O(N)$ solves with the operator $1/2 I - S_{k}'$ are required, and fast direct solvers tend to be the method of choice in this setting. In particular we use the recursive skeletonization approach~\cite{rskel} for obtaining a compressed approximation of the inverse.

Finally, the evaluation of $\delta \mathcal{N}$ is straightforward as it requires an integral of a smooth function on the boundary $\partial \Omega$, once $\delta \sigma$, and $\delta k$ are available. 

%% file: sections/numerical_experiments.tex
\section{Numerical Experiments}\label{sec:numerics}
We approximate domains in $\mathcal{D}_0$ by polygons parameterized in polar coordinates.
For $N\ge 2$, we define equally spaced angles $\theta_i:=\frac{i-1}{N-1}\pi$ for $i=1,\dots,N$.
Given radii $r=(r_1,\dots,r_N)\in\mathbb{R}_{>0}^N$, we set
  \begin{align}
    p_i(r):=(x_i,y_i):=(r_i\cos\theta_i,r_i\sin\theta_i),\quad\quad i=1,\dots,N,
  \end{align}
and let $\tilde{\Omega}_N(r)$ be the $N$-gon obtained by connecting $p_1(r),\dots,p_N(r)$ in order and closing the boundary with the segment $[p_N(r),p_1(r)]\subset\{y=0\}$. The domain $\Omega_{N}(r)$ is then obtained by rounding $\tilde{\Omega}_{N}(r)$ following the procedure in~\cite{epstein2016smoothed} with a minor modification discussed below. 
\begin{remark}
The polygonal domain is rounded to simplify the computation of the Dirichlet eigenvalue, which is obtained by finding the roots of a discretization of the Fredholm determinant of $I - 2S_{k}'$~\cite{zhao2015robust}. The Fredholm determinant is well-defined for second kind Fredholm operators of trace class to which $I-2S_{k}'$ belongs when the boundary is at least $C^{2}$. The numerical procedure discussed here can be extended to polygonal domains by instead computing the eigenvalue via the zeros of $f(k) = 1/(v \cdot (I-2S_{k}')^{-1}u)$ for some smooth functions $u$ and $v$ defined on $\partial \Omega$, see~\cite{cheng2004fast,lai2018second}, for example.
\end{remark}

The rounding procedure in~\cite{epstein2016smoothed} replaces the polygon in the vicinity of every vertex by an affine transformation of a smooth approximation of the graph of $f(x) = |x|$ which agrees with $|x|$ outside the interval $[-h,h]$. The rounding parameter $h$ at vertex $p_{i}$ is chosen to be $\alpha \min(|p_{i} - p_{i+1}|, |p_{i} - p_{i-1}|)$ for some $\alpha <1/2$. However, the resulting boundary deformations $V$ corresponding to the $x$ and $y$ perturbations of the vertex $p_{i}$ will be non-smooth whenever $|p_{i} - p_{i+1}| = |p_{i} - p_{i-1}|$. We instead use a smooth version of the $\min$ function  given by $$\textrm{smin}(a,b) = \frac{1}{\log{(\exp{(1/a)} + \exp{(1/b)} -1)}}\,,$$
which satisfies $\min(a,b)/2 \leq \textrm{smin}(a,b) \leq \min(a,b)$. In all the experiments in this section $\alpha = 0.1$.

We define the discretized objective
  \begin{align}
    f_N:\mathbb{R}_{>0}^N\to\mathbb{R},\quad\quad f_N(r):=\mathcal{F}(\Omega_N(r)),
  \end{align}
where $\mathcal{F}$ is the functional in \eqref{eqn: definition of functional}.
We discuss details of numerical evaluations of the objective in Section~\ref{sec:objective}.
For a fixed $N$, we maximize $f_N$ by gradient ascent starting from some initialization to find the optimal rounded $N$-gon.
The gradient is assembled from shape derivatives of the individual terms in $\mathcal{F}$ and then converted to derivatives with respect to the parameters $r_i$, as described in Section~\ref{sec:gradients}.

\subsection{Objective computation}\label{sec:objective}
For each parameter vector $r$, we compute the first Dirichlet eigenvalue $\lambda_1$ of $\Omega_N(r)$ together with the value of the functional $\mathcal{F}(\Omega_N(r))$.
Writing $\lambda_1=k_1^2$, our computation is based on the boundary integral formulation described in Section~\ref{sec:gradients}.
In particular, if $u_1$ is the first eigenfunction and $\sigma:=-\partial_n u_1$ denotes its inward normal derivative (with $n$ the outward unit normal), then $\sigma$ lies in the nullspace of the second-kind operator $\frac{1}{2} I-S'_{k_1}$.
Thus $k_1$ is characterized by the existence of a nontrivial solution of $\big(\frac{1}{2} I-S'_k\big)[\sigma]=0$.
Since $S'_k$ is compact for smooth $\Gamma=\partial\Omega$, the corresponding Fredholm determinant vanishes precisely at Dirichlet eigenvalues.
Using this fact, we locate $k_1$ by searching for the smallest positive root of this determinant~\cite{zhao2015robust,bornemann2010numerical}.

To do so numerically, we use the MATLAB software package chunkIE~\cite{chunkIE}.
It rounds the corners of a polygon and splits the rounded polygon's boundary into panels.
Each panel is represented by 16th-order Gauss--Legendre nodes.
We discretize $\frac{1}{2} I-S'_k$ by a high-order Nystr\"om method, evaluate the resulting determinant at Chebyshev points in a bracket $[a,b]$ for $k_1$, and extract the smallest real root using a Chebyshev interpolant~\cite{platte2010chebfun}.
To accelerate the repeated determinant evaluations, we use the recursive skeletonization routines from FLAM~\cite{ho2020flam} and \texttt{fmm2d}~\cite{greengard1987fast} via chunkIE.

Once $k_1$ is known, we compute a corresponding discrete nontrival null vector $\sigma$, then normalize it using the Rellich identity \eqref{eqn: uL2_from_sigma},
  \begin{align}
    \mathcal{N}(\Omega):=\|u_1\|_{L^2(\Omega)}^2
    = \frac{1}{2k_1^2}\int_\Gamma (x\cdot n)\,\sigma(x)^2\,{\rm d}s_x,
  \end{align}
which we approximate by boundary quadrature.
Finally, we evaluate $\sigma$ at the distinguished boundary point $x_\ast=(0,0)$ by high-order interpolation on the panel containing $x_\ast$ and compute the objective via \eqref{eqn: objective_from_sigma}:
  \begin{align}
    f_N(r)=\mathcal{F}(\Omega_N(r))=\frac{\sigma(x_\ast)}{k_1^2\sqrt{\mathcal{N}(\Omega_N(r))}}.
  \end{align}

\subsection{Results}
We first fix $N=16$ and maximize $f_{16}$ from several initial polygons, as shown in Figure~\ref{fig:3init}.
  \begin{figure}[h]
    \centering
    \includegraphics[width=1.0\textwidth]{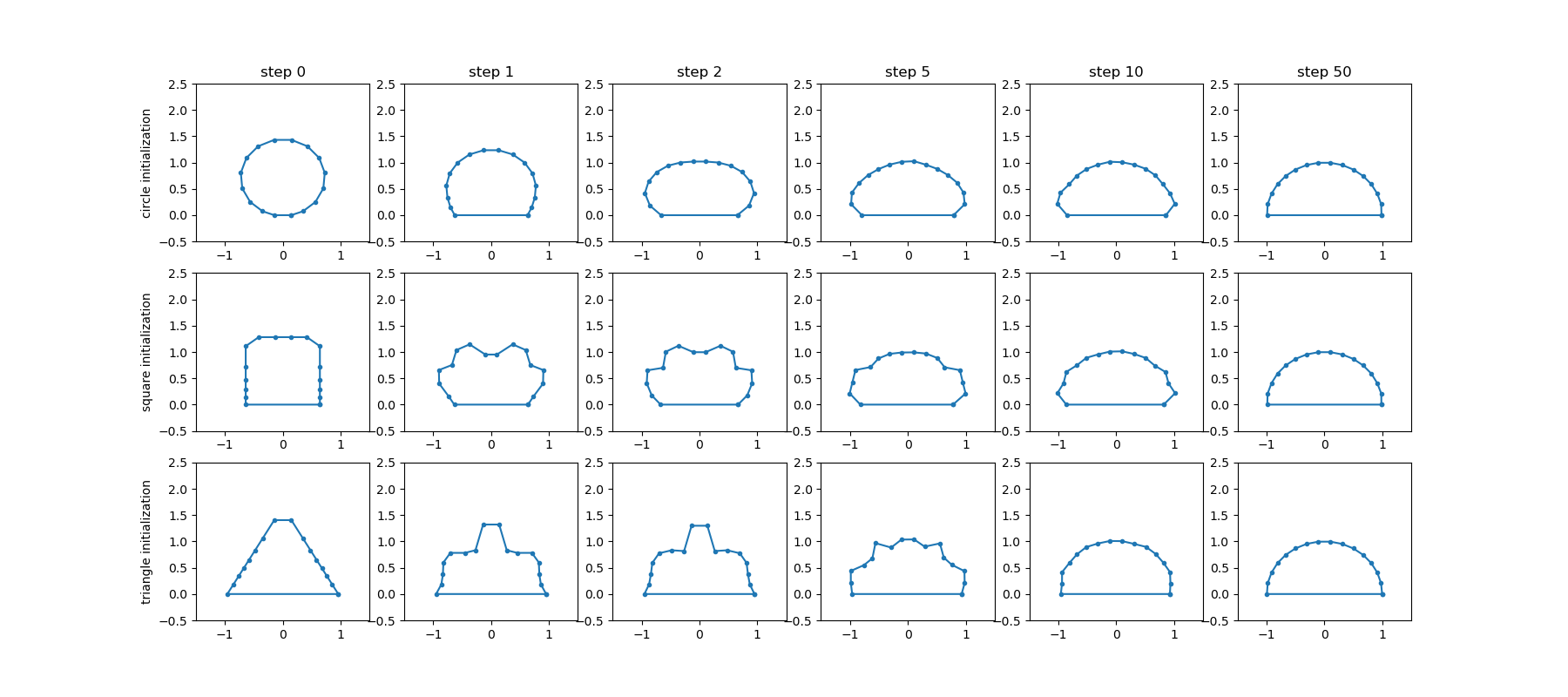} 
    \caption{Optimization results with $N=16$: Each row represents a different initialization. Going from top to bottom, we have circle, square and triangle initializations. Each column is a different time step during optimization.}
    \label{fig:3init}
  \end{figure}
To test robustness under discretization refinement, we increase the number of vertices during the optimization.
Starting from $N=8$, we alternate between optimizing at fixed $N$ until $\|\nabla f_N(r)\|_2\le \eta$ and refining the polygon by inserting additional vertices. We terminate refinement once $N$ reaches a prescribed target $N_{\mathrm{target}}$.
See Algorithm~\ref{alg:vertex_refinement} for details of our vertex refinement strategy.
\input{tables/vertex_refinement}
Algorithm~\ref{alg:gradient_method} details the gradient ascent procedure, including the choice of step size with two local evaluations as well as fallback strategy.
\input{tables/gradient_descent}
See Figure~\ref{fig:c100_sample_by_n} for details on a large-scale experiment using \vertexrefinement, which result with a semidisk-like shape with $N=113$ vertices.
We use a large enough $K$ so that \gradientascent \,is not limited by the maximum number of iterations.
  \begin{figure}[h]
    \centering
    \includegraphics[width=1.0\textwidth]{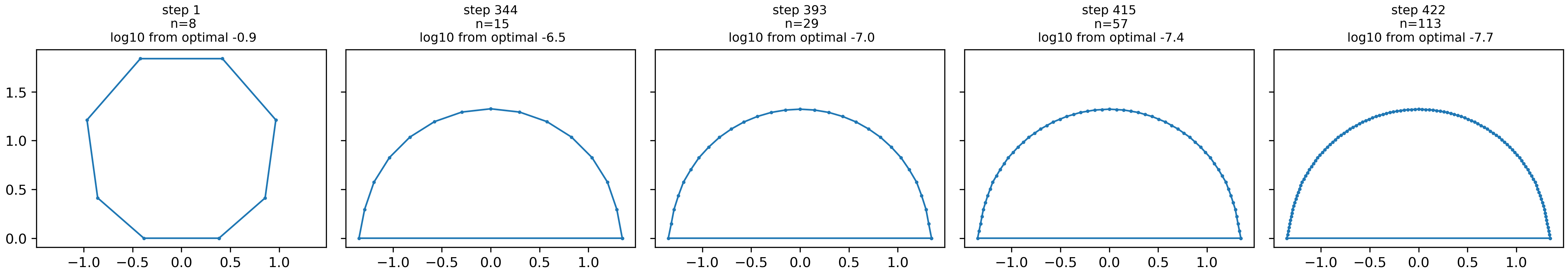} 
    \caption{A large scale experiment that starts with $N=8$ vertices, and gradually increases the number of vertices exceeds $N_{\rm{target}}=100$ with gradient threshold set to $\eta=5\times 10^{-6}$. Under vertex addition strategy in \vertexrefinement, the resulting polygon has $N=113$ vertices.}
    \label{fig:c100_sample_by_n}
  \end{figure}
\FloatBarrier

%% file: tables/vertex_refinement.tex
\begin{algorithm}[h]
  \caption{\vertexrefinement}
  \label{alg:vertex_refinement}
  \begin{algorithmic}[1]
    \REQUIRE target number of vertices $N_{\mathrm{target}}$, gradient tolerance $\eta>0$, max iterations $K$
    \STATE Initialize $N\gets 8$ and radii $r\in\mathbb{R}^N_{>0}$.
    \WHILE{true}
      \STATE $r\gets \gradientascent(r,\eta,K)$.
      \IF{$N\ge N_{\mathrm{target}}$}
        \STATE \textbf{break}
      \ENDIF
      \STATE Let $\theta_i:=\frac{i-1}{N-1}\pi$ and $p_i:=(r_i\cos\theta_i,r_i\sin\theta_i)$ for $i=1,\dots,N$.
      \STATE Set $\tilde r_{2i-1}\gets r_i$ for $i=1,\dots,N$.
      \FOR{$i=1,\dots,N-1$}
        \STATE $\tilde p_{2i}\gets[p_i,p_{i+1}]\cap \{(\rho\cos\frac{\theta_i+\theta_{i+1}}2,\rho\sin\frac{\theta_i+\theta_{i+1}}2):\rho>0\}$.
        \STATE $\tilde r_{2i}\gets\|\tilde p_{2i}\|_2$.
      \ENDFOR
      \STATE Set $r\gets \tilde r$ and $N\gets 2N-1$.
    \ENDWHILE
  \end{algorithmic}
\end{algorithm}

%% file: tables/gradient_descent.tex
\begin{algorithm}[h]
  \caption{\gradientascent}
  \label{alg:gradient_method}
  \begin{algorithmic}[1]
    \REQUIRE initialization $r^{(0)}\in\mathbb{R}^N_{>0}$, gradient tolerance $\eta>0$, max iterations $K$
    \STATE \textbf{Parameter:} step size $h=10^{-2}$
    \FOR{$k=0,1,\dots,K-1$}
      \STATE $g\gets \nabla f_N(r^{(k)})$
      \IF{$\|g\|_2\le \eta$}
        \STATE{\textbf{break}}
      \ENDIF
      \STATE Set direction $d\gets g/\|g\|_2$
      \STATE $f_0\gets f_N(r^{(k)})$
      \STATE $f_+\gets f_N(r^{(k)}+h d)$ and $f_-\gets f_N(r^{(k)}-h d)$
      \STATE $a\gets \frac{f_+ + f_- - 2f_0}{2h^2}$ and $b\gets \frac{f_+ - f_-}{2h}$
      \STATE $\alpha^\ast\gets -\frac{b}{2a}$
      \IF{$\alpha^\ast<0$}
        \STATE $r^{(k+1)}\gets r^{(k)}+g$
      \ELSE
        \STATE $r^{(k+1)}\gets r^{(k)}+\alpha^\ast d$
      \ENDIF
    \ENDFOR
    \RETURN updated radii $r$
  \end{algorithmic}
\end{algorithm}

%% file: sections/conclusion.tex
\section{Conclusion and future work}\label{sec:conclusion}
We studied the size of the normal derivative of the first Dirichlet eigenfunction on the boundary of a convex planar domain.
We proved a universal upper bound of the form $\|\partial_n u_1\|_{L^\infty(\partial\Omega)}\le C\,\lambda_1$ (Theorem~\ref{thm: loose upper bound}).
Motivated by numerical experiments, we formulated Conjecture~\ref{conj: semicircle}, predicting that the semidisk maximizes the scale-invariant quantity $\mathcal{G}(\Omega)=\|\partial_n u_1\|_{L^\infty(\partial\Omega)}/\lambda_1$ among bounded convex planar domains.
As a partial result, we proved that the semidisk is a critical point of the pointwise functional $\mathcal{F}$ under a natural class of perturbations of the semicircular arc (Theorem~\ref{thm: critical point}).

We also described a boundary integral formulation and derived the shape derivatives needed to assemble gradients of the discretized objective efficiently.
The resulting gradient ascent experiments, across a range of initializations and discretization refinements, consistently converged to semidisk-like shapes, providing further evidence for Conjecture~\ref{conj: semicircle}.

There are several natural directions for future work.
First, one would like to establish existence of an extremizer for $\mathcal{G}$ in a suitable topology.
Since $\mathcal{G}$ is scale-invariant, one may restrict to a class of convex domains that is compact in the Hausdorff metric; by the Blaschke selection theorem~\cite{blaschke1916kreis} such classes are sequentially compact.
Thus, a key remaining step is to prove an appropriate continuity (or upper semi-continuity) statement for $\mathcal{G}$ under Hausdorff convergence.
Second, it would be interesting to obtain a more rigorous characterization of the maximizer by proving local optimality of the semidisk (e.g. via second-variation analysis and convexity constraints).
Finally, the same methodology can be applied to related optimization problems, such as other norms of $\partial_n u_1$ on $\partial\Omega$, different eigenmodes, or alternative boundary conditions.

%% file: main.bib
@article{lai2018second,
  title={Second kind integral equation formulation for the mode calculation of optical waveguides},
  author={Lai, Jun and Jiang, Shidong},
  journal={Applied and Computational Harmonic Analysis},
  volume={44},
  number={3},
  pages={645--664},
  year={2018},
  publisher={Elsevier}
}

@article{cheng2004fast,
  title={Fast, accurate integral equation methods for the analysis of photonic crystal fibers I: Theory},
  author={Cheng, H and Crutchfield, W and Doery, M and Greengard, Leslie},
  journal={Optics Express},
  volume={12},
  number={16},
  pages={3791--3805},
  year={2004},
  publisher={OSA}
}

@article{rskel,
  title={A fast direct solver for structured linear systems by recursive skeletonization},
  author={Ho, Kenneth L and Greengard, Leslie},
  journal={SIAM Journal on Scientific Computing},
  volume={34},
  number={5},
  pages={A2507--A2532},
  year={2012},
  publisher={SIAM}
}

@incollection{platte2010chebfun,
  title={Chebfun: a new kind of numerical computing},
  author={Platte, Rodrigo B and Trefethen, Lloyd N},
  booktitle={Progress in industrial mathematics at ECMI 2008},
  pages={69--87},
  year={2010},
  publisher={Springer}
}

@article{greengard1987fast,
  title={A fast algorithm for particle simulations},
  author={Greengard, Leslie and Rokhlin, Vladimir},
  journal={Journal of computational physics},
  volume={73},
  number={2},
  pages={325--348},
  year={1987},
  publisher={Elsevier}
}

@article{biswas2017location,
  title={Location of Maximizers of Eigenfunctions of Fractional {S}chr{\"o}dinger’s Equations},
  author={Biswas, Anup},
  journal={Mathematical Physics, Analysis and Geometry},
  volume={20},
  number={4},
  pages={25},
  year={2017},
  publisher={Springer}
}

@article{rachh2018location,
  title={On the location of maxima of solutions of {S}chrödinger's equation},
  author={Rachh, Manas and Steinerberger, Stefan},
  journal={Communications on Pure and Applied Mathematics},
  volume={71},
  number={6},
  pages={1109--1122},
  year={2018},
  publisher={Wiley Online Library}
}

@article{mossino1983generalization,
  title={A generalization of the {P}ayne--{R}ayner isoperimetric inequality},
  author={Mossino, J},
  journal={Bolletino della Unione Mathematica Italiana},
  volume={2},
  number={3},
  pages={335--342},
  year={1983},
}

@article{alvino1998properties,
  title={On the properties of some nonlinear eigenvalues},
  author={Alvino, Angelo and Ferone, Vincenzo and Trombetti, Guido},
  journal={SIAM Journal on Mathematical Analysis},
  volume={29},
  number={2},
  pages={437--451},
  year={1998},
  publisher={SIAM}
}

@article{wang2010isoperimetric,
  title={Isoperimetric bounds for the first eigenvalue of the Laplacian},
  author={Wang, Qiaoling and Xia, Changyu},
  journal={Zeitschrift f{\"u}r angewandte Mathematik und Physik},
  volume={61},
  number={1},
  pages={171--175},
  year={2010},
  publisher={Springer}
}

@article{freitas2015first,
  title={The first {R}obin eigenvalue with negative boundary parameter},
  author={Freitas, Pedro and Krej{\v{c}}i{\v{r}}{\'\i}k, David},
  journal={Advances in Mathematics},
  volume={280},
  pages={322--339},
  year={2015},
  publisher={Elsevier}
}

@article{bossel1986membranes,
  title={Membranes {\'e}lastiquement li{\'e}es: extension du th{\'e}or{\`e}me de {Rayleigh--Faber--Krahn} et de l'in{\'e}galit{\'e} de {C}heeger},
  author={Bossel, M-H},
  journal={Comptes rendus de l'Acad{\'e}mie des sciences. S{\'e}rie 1, Math{\'e}matique},
  volume={302},
  number={1},
  pages={47--50},
  year={1986}
}

@article{zhao2015robust,
  title={Robust and efficient solution of the drum problem via {N}ystr\"om approximation of the {F}redholm determinant},
  author={Zhao, Lin and Barnett, Alex},
  journal={SIAM Journal on Numerical Analysis},
  volume={53},
  number={4},
  pages={1984--2007},
  year={2015},
  publisher={SIAM}
}

@article{epstein2016smoothed,
  title={Smoothed corners and scattered waves},
  author={Epstein, Charles L and O'Neil, Michael},
  journal={SIAM Journal on Scientific Computing},
  volume={38},
  number={5},
  pages={A2665--A2698},
  year={2016},
  publisher={SIAM}
}

@article{daners2006faber,
  title={A {Faber--Krahn} inequality for {R}obin problems in any space dimension},
  author={Daners, Daniel},
  journal={Mathematische Annalen},
  volume={335},
  number={4},
  pages={767--785},
  year={2006},
  publisher={Springer}
}

@book{rayleigh1896theory,
  title={The theory of sound},
  author={Rayleigh, John William Strutt Baron},
  volume={2},
  year={1896},
  publisher={Macmillan}
}

@article{ashbaugh1992sharp,
  title={A sharp bound for the ratio of the first two eigenvalues of {D}irichlet {L}aplacians and extensions},
  author={Ashbaugh, Mark S and Benguria, Rafael D},
  journal={Annals of Mathematics},
  volume={135},
  number={3},
  pages={601--628},
  year={1992},
  publisher={JSTOR}
}

@article{kennedy2009isoperimetric,
  title={An isoperimetric inequality for the second eigenvalue of the {L}aplacian with {R}obin boundary conditions},
  author={Kennedy, James},
  journal={Proceedings of the American Mathematical Society},
  volume={137},
  number={2},
  pages={627--633},
  year={2009}
}

@techreport{polya1955characteristic,
  title={On the characteristic frequencies of a symmetric membrane},
  author={P{\'o}lya, G},
  year={1955}
}

@article{payne1961some,
  title={Some isoperimetric inequalities for membrane frequencies and torsional rigidity},
  author={Payne, Lawrence Edward and Weinberger, Hans F},
  journal={Journal of Mathematical Analysis and Applications},
  volume={2},
  number={2},
  pages={210--216},
  year={1961},
  publisher={Elsevier}
}

@article{polya2016isoperimetric,
  title={Isoperimetric Inequalities in Mathematical Physics},
  author={P{\'o}lya, George and Szeg{\"o}, G{\'a}bor},
  year={1951},
  publisher={Princeton University Press},
  isbn=0691079889,
}

@article{brakhage1965dirichletsche,
  title={{{\"U}ber das Dirichletsche Au{\ss}enraumproblem f{\"u}r die Helmholtz sche Schwingungsgleichung}},
  author={Brakhage, Helmut and Werner, Peter},
  journal={Archiv der Mathematik},
  volume={16},
  number={1},
  pages={325--329},
  year={1965},
  publisher={Springer}
}

@article{potthast1996frechet,
  title={{Fr{\'e}chet differentiability of the solution to the acoustic Neumann scattering problem with respect to the domain}},
  author={Potthast, Roland},
  year={1996},
  publisher={Walter de Gruyter, Berlin/New York Berlin, New York}
}

@article{hohage1997logarithmic,
  title={Logarithmic convergence rates of the iteratively regularized {Gauss--Newton} method for an inverse potential and an inverse scattering problem},
  author={Hohage, Thorsten},
  journal={Inverse problems},
  volume={13},
  number={5},
  pages={1279},
  year={1997},
  publisher={IOP Publishing}
}

@book{colton2013integral,
  title={Integral equation methods in scattering theory},
  author={Colton, David and Kress, Rainer},
  year={2013},
  publisher={SIAM}
}

@article{liu2020regularized,
  title={Regularized {N}ewton iteration method for a penetrable cavity with internal measurements in inverse scattering problem},
  author={Liu, Lihan and Cai, Jingqiu and Steve Xu, Yongzhi},
  journal={Mathematical Methods in the Applied Sciences},
  volume={43},
  number={5},
  pages={2665--2678},
  year={2020},
  publisher={Wiley Online Library}
}

@article{potthast1994frechet,
  title={{F}r{\'e}chet differentiability of boundary integral operators in inverse acoustic scattering},
  author={Potthast, Roland},
  journal={Inverse Problems},
  volume={10},
  pages={431--447},
  year={1994}
}

@article{hohage1998newton,
  title={A {N}ewton-type method for a transmission problem in inverse scattering},
  author={Hohage, T and Schormann, C},
  journal={Inverse Problems},
  volume={14},
  pages={1207--1227},
  year={1998}
}

@article{brasco2012sharp,
  title={Sharp stability of some spectral inequalities},
  author={Brasco, Lorenzo and Pratelli, Aldo},
  journal={Geometric and Functional Analysis},
  volume={22},
  number={1},
  pages={107--135},
  year={2012},
  publisher={Springer}
}

@article{brasco2015faber,
  title={{Faber--Krahn} inequalities in sharp quantitative form},
  author={Brasco, Lorenzo and De Philippis, Guido and Velichkov, Bozhidar},
  year={2015}
}

@article{krahn1925rayleigh,
  title={{{\"U}ber eine von Rayleigh formulierte Minimaleigenschaft des Kreises}},
  author={Krahn, Edgar},
  journal={Mathematische Annalen},
  volume={94},
  number={1},
  pages={97--100},
  year={1925},
  publisher={Springer}
}

@article{szego1954inequalities,
  title={Inequalities for certain eigenvalues of a membrane of given area},
  author={Szeg{\"o}, G{\'a}bor},
  journal={Journal of Rational Mechanics and Analysis},
  volume={3},
  pages={343--356},
  year={1954},
  publisher={JSTOR}
}

@article{weinberger1956isoperimetric,
  title={An isoperimetric inequality for the N-dimensional free membrane problem},
  author={Weinberger, Hans F},
  journal={Journal of Rational Mechanics and Analysis},
  volume={5},
  number={4},
  pages={633--636},
  year={1956},
  publisher={JSTOR}
}

@book{henrot2006extremum,
  title={Extremum problems for eigenvalues of elliptic operators},
  author={Henrot, Antoine},
  year={2006},
  publisher={Springer}
}

@article{bornemann2010numerical,
  title={On the numerical evaluation of {F}redholm determinants},
  author={Bornemann, Folkmar},
  journal={Mathematics of Computation},
  volume={79},
  number={270},
  pages={871--915},
  year={2010}
}

@article{ho2020flam,
  title={{FLAM}: Fast linear algebra in {MATLAB}---Algorithms for hierarchical matrices},
  author={Ho, Kenneth L},
  journal={Journal of Open Source Software},
  volume={5},
  number={51},
  pages={1906},
  year={2020}
}

@book{delfour2011shapes,
  title={Shapes and geometries: metrics, analysis, differential calculus, and optimization},
  author={Delfour, Michel C and Zol{\'e}sio, J-P},
  year={2011},
  publisher={SIAM}
}

@article{keady1993elastic,
  title={The elastic torsion problem: solutions in convex domains},
  author={Keady, Grant and McNabb, Alex},
  journal={NZ Journal of Mathematics},
  volume={22},
  number={43-64},
  pages={30},
  year={1993},
  publisher={Citeseer}
}

@article{sperb1981maximum,
  title={Maximum principles and their applications},
  author={Sperb, Ren{\'e} Peter},
  year={1981}
}

@book{blaschke1916kreis,
  title={Kreis und kugel},
  author={Blaschke, Wilhelm},
  year={1916},
  publisher={Veit Leipzig}
}

@article{hoskins2021towards,
  title={Towards optimal gradient bounds for the torsion function in the plane},
  author={Hoskins, Jeremy G and Steinerberger, Stefan},
  journal={The Journal of Geometric Analysis},
  volume={31},
  pages={7812--7841},
  year={2021},
  publisher={Springer}
}

@article{moler1968bounds,
  title={Bounds for eigenvalues and eigenvectors of symmetric operators},
  author={Moler, CB and Payne, LE},
  journal={SIAM Journal on Numerical Analysis},
  volume={5},
  number={1},
  pages={64--70},
  year={1968},
  publisher={SIAM}
}

@article{grebenkov2013geometrical,
  title={Geometrical structure of {L}aplacian eigenfunctions},
  author={Grebenkov, Denis S and Nguyen, B-T},
  journal={SIAM Rev.},
  volume={55},
  number={4},
  pages={601--667},
  year={2013},
}

@article {hassell2002upper,
    AUTHOR = {Hassell, Andrew and Tao, Terence},
     TITLE = {Upper and lower bounds for normal derivatives of {D}irichlet
              eigenfunctions},
   JOURNAL = {Math. Res. Lett.},
  FJOURNAL = {Mathematical Research Letters},
    VOLUME = {9},
      YEAR = {2002},
    NUMBER = {2--3},
     PAGES = {289--305},
      ISSN = {1073-2780},
   MRCLASS = {58J50 (35P15)},
  MRNUMBER = {MR1909646 (2003k:58047)},
MRREVIEWER = {E. Lami Dozo},
}

@article{payne1973mean,
  title={On the mean value of the fundamental mode in the fixed membrane problem},
  author={Payne, LE and Stakgold, I},
  journal={Applicable Analysis},
  volume={3},
  number={3},
  pages={295--306},
  year={1973},
  publisher={Taylor \& Francis}
}

@article{kohler1977premiere,
  title={{Sur la premi{\`e}re fonction propre d'une membrane: une extension {\`a} N dimensions de l'in{\'e}galit{\'e} isop{\'e}rim{\'e}trique de Payne-Rayner}},
  author={Kohler-Jobin, Marie-Th{\'e}r{\`e}se},
  journal={Zeitschrift f{\"u}r angewandte Mathematik und Physik ZAMP},
  volume={28},
  pages={1137--1140},
  year={1977},
  publisher={Springer}
}

@article{payne1973some,
  title={Some isoperimetric norm bounds for solutions of the {H}elmholtz equation},
  author={Payne, Lawrence E and Rayner, Margaret E},
  journal={Zeitschrift f{\"u}r angewandte Mathematik und Physik ZAMP},
  volume={24},
  pages={105--110},
  year={1973},
  publisher={Springer}
}

@article{payne1972isoperimetric,
  title={An isoperimetric inequality for the first eigenfunction in the fixed membrane problem},
  author={Payne, Lawrence E and Rayner, Margaret E},
  journal={Zeitschrift f{\"u}r angewandte Mathematik und Physik ZAMP},
  volume={23},
  pages={13--15},
  year={1972},
  publisher={Springer}
}

@book{dlmf,
title="{NIST} Handbook of Mathematical Functions",
editor={Frank W. J. Olver and Daniel W. Lozier and Ronald F. Boisvert and Charles W. Clark},
publisher={Cambridge University Press},
year=2010,
note={\url{http://dlmf.nist.gov}},
}

@book{mclean2000strongly,
  title={Strongly elliptic systems and boundary integral equations},
  author={McLean, William Charles Hector},
  year={2000},
  publisher={Cambridge university press}
}

@article{grin10,
  title={{H}adamard’s formula inside and out},
  author={Grinfeld, P},
  journal={Journal of optimization theory and applications},
  volume={146},
  pages={654--690},
  year={2010},
  publisher={Springer}
}

@article{levitan1952asymptotic,
  title={On the asymptotic behavior of the spectral function of a self-adjoint differential equation of the second order},
  author={Levitan, Boris Moiseevich},
  journal={Izvestiya Rossiiskoi Akademii Nauk. Seriya Matematicheskaya},
  volume={16},
  number={4},
  pages={325--352},
  year={1952},
  publisher={Russian Academy of Sciences, Steklov Mathematical Institute of Russian~…}
}

@article{avakumovic1956eigenfunktionen,
  title={{{\"u}ber die Eigenfunktionen auf geschlossenen Riemannschen Mannigfaltigkeiten}},
  author={Avakumovi{\'c}, Vojislav G},
  journal={Mathematische Zeitschrift},
  volume={65},
  number={1},
  pages={327--344},
  year={1956},
  publisher={Springer}
}

@article{hormander1968riesz,
  title={On the {R}iesz means of spectral functions and eigenfunction expansions for elliptic differential operators},
  author={H{\"o}rmander, Lars Valter},
  journal={Matematika},
  volume={12},
  number={5},
  pages={91--130},
  year={1968}
}

@article{van2000norm,
  title={On the {L}-infinity norm of the first eigenfunction of the {D}irichlet {L}aplacian},
  author={Van Den Berg, M},
  journal={Potential Analysis},
  volume={13},
  number={4},
  pages={361--366},
  year={2000},
  publisher={Springer}
}

@article{Weyl1911,
author = {Weyl, H.},
journal = {Nachrichten von der Gesellschaft der Wissenschaften zu Göttingen, Mathematisch-Physikalische Klasse},
pages = {110-117},
title = {{Ueber die asymptotische Verteilung der Eigenwerte}},
url = {http://eudml.org/doc/58792},
volume = {1911},
year = {1911},
}

@article{kac1966can,
  title={Can one hear the shape of a drum?},
  author={Kac, Mark},
  journal={The American Mathematical Monthly},
  volume={73},
  number={4P2},
  pages={1--23},
  year={1966},
  publisher={Taylor \& Francis}
}

@misc{chunkIE,
  title = {{chunkIE}: a {MATLAB} integral equation toolbox},
  author = {Askham, Travis and Rachh, Manas and O'Neil, Michael and Hoskins, Jeremy and Fortunato, Daniel and Jiang, Shidong and Fryklund, Fredrik and Goodwill, Tristan and Wang, Hai Yang and Zhu, Hai},
  year = {2024},
  note = {Version 1.0.0, available at \url{https://chunkie.readthedocs.io/}},
  howpublished = {GitHub repository: \url{https://github.com/fastalgorithms/chunkie}},
  month = jun
}

@article{SZ02,
author = {Christopher D. Sogge and Steve Zelditch},
title = {{Riemannian manifolds with maximal eigenfunction growth}},
volume = {114},
journal = {Duke Mathematical Journal},
number = {3},
publisher = {Duke University Press},
pages = {387 -- 437},
year = {2002},
doi = {10.1215/S0012-7094-02-11431-8},
}

@incollection{Z03,
     author = {Zelditch, Steve},
     title = {Billiards and boundary traces of eigenfunctions},
     booktitle = {},
     series = {Journ\'ees \'equations aux d\'eriv\'ees partielles},
     eid = {15},
     pages = {1--22},
     year = {2003},
     publisher = {Universit\'e de Nantes},
     doi = {10.5802/jedp.629},
     url = {https://www.numdam.org/articles/10.5802/jedp.629/}
}

@article{SZ17,
author = {Christopher D. Sogge and Steve Zelditch},
title = {Sup norms of {C}auchy data of eigenfunctions on manifolds with concave boundary},
journal = {Communications in Partial Differential Equations},
volume = {42},
number = {8},
pages = {1249--1289},
year = {2017},
doi = {10.1080/03605302.2017.1349147},
}

@misc{hotspots,
      title={Convex sets can have interior hot spots}, 
      author={Jaume de Dios Pont},
      year={2024},
      eprint={2412.06344},
      archivePrefix={arXiv},
      primaryClass={math.AP},
      url={https://arxiv.org/abs/2412.06344}, 
}

@misc{fmm2d,
title={{FMM2D}: fast multipole library in two dimensions. {V}ersion 1.1.0},
author={Rachh, Manas and Greengard, Leslie and Gimbutas, Zydrunas},
note={\url{https://github.com/flatironinstitute/fmm2d}},
year={2024},
}

@article{Osting10,
title = {Optimization of spectral functions of {D}irichlet--{L}aplacian eigenvalues},
journal = {J. Comput. Phys.},
volume = {229},
number = {22},
pages = {8578--90},
year = {2010},
doi = {10.1016/j.jcp.2010.07.040},
author = {Braxton Osting},
}

@article{oudet04,
author={Oudet, {\'E}douard},
title={Numerical minimization of eigenmodes of a membrane with respect to the domain},
journal={ESAIM: Control, Optimisation and Calculus of Varitions},
volume={10},
year={2004},
pages={315--335},
}

@article{antunes12,
journal={J. Optim. Theory Appl.},
year=2012,
volume=154,
pages={235--257},
doi={10.1007/s10957-011-9983-3},
title={Numerical Optimization of Low Eigenvalues of the {D}irichlet and {N}eumann {L}aplacians},
author={Pedro R. S. Antunes and Pedro Freitas},
}

@article {rellich,
    AUTHOR = {Rellich, Franz},
     TITLE = {Darstellung der {E}igenwerte von {$\Delta u+\lambda u=0$}
              durch ein {R}andintegral},
   JOURNAL = {Math. Z.},
  FJOURNAL = {Mathematische Zeitschrift},
    VOLUME = {46},
      YEAR = {1940},
     PAGES = {635--636}
}
